\documentclass[12pt,reqno,svgnames,dvipsnames]{amsart}

\usepackage{amssymb}
\usepackage{enumerate}
\usepackage[pagebackref,colorlinks,linkcolor=red,citecolor=blue,urlcolor=blue,hypertexnames=true]{hyperref}
\usepackage{graphicx,tikz}
\usetikzlibrary{shapes,patterns,calc,snakes}

\usepackage{titletoc}

\DeclareMathOperator{\sign}{sign}
\DeclareMathOperator{\rank}{rank}

\newcommand{\NCSOStools}{\href{http://ncsostools.fis.unm.si/}{{\tt NCSOStools}} }
\newcommand{\NCSOStoolz}{\href{http://ncsostools.fis.unm.si/}{{\tt NCSOStools}}}

\newcommand{\NCAlgebra}{\href{http://www.math.ucsd.edu/~ncalg/}{{\tt NCAlgebra}} }
\newcommand{\NCAlgebrah}{\href{http://www.math.ucsd.edu/~ncalg/}{{\tt NCAlgebra}}}

\makeatletter
\newcommand{\mycontentsbox}{%
{\centerline{NOT FOR PUBLICATION}
\small\tableofcontents}}
\def\enddoc@text{\ifx\@empty\@translators \else\@settranslators\fi
\ifx\@empty\addresses \else\@setaddresses\fi
\newpage\mycontentsbox}
\makeatother

\newcommand{\dom}{\operatorname{dom}}

\newcommand{\ttpp}[3]{\cT_{#1}({#2},{#3})}   

\def\cB{ \ensuremath{\mathcal B} }
\def\cC{\ensuremath {\mathcal C} }
\def\cD{ \ensuremath{\mathcal D} }
\def\cE{ \ensuremath{\mathcal E} }
\def\cF{ \ensuremath{\mathcal F} }

\def\cH{ \ensuremath{\mathcal H} }

\def\cM{\ensuremath{\mathcal M}}
\def\cN{\ensuremath {\mathcal N} }
\def\cO{ \ensuremath{\mathcal O} }
\def\cP{\ensuremath{\mathcal P}}

\def\cR{\ensuremath{ \mathcal R }}

\def\cT{\ensuremath{\mathcal T}}
\def\cU{\ensuremath{\mathcal U}}
\def\cV{\ensuremath{\mathcal V}}

\def\cY{\ensuremath{\mathcal Y}}
\def\cX{\ensuremath{\mathcal X}}

\newcommand{\NN}{{\ensuremath\mathbb N}}
\def\RR{{\ensuremath{\mathbb R}} }
\def\SS{{\ensuremath{\mathbb S}} }

\def\BB{{\ensuremath{\mathbb B}}}

\def\SR{\SS }

\def\SRdd{\SS^{d \times d} }
\def\SRnn{\SS^{n \times n} }

\def\Snn{\SS^{n\times n}}

\def\Mdd{\RR^{d \times d^\prime} }

\def\Mnns{\RR^{\ell\times \ell}}

\def\gtupn{(\SRnn)^g}
\def\twotupn{(\SRnn)^2}
\def\twotuptwo{(\SR^{2\times 2})^2}
\def\gtupm{(\SS^{m\times m})^g}

\newcommand{\ncalg}{\vspace{.5\baselineskip}\noindent{\it NCAlgebra Command: }}
\newcommand{\ncsos}{\vspace{.5\baselineskip}\noindent{\it NCSOStools Command: }}

\def\ss{\intercal}
\def\Ms{\intercal}

\def\x{x}
\newcommand{\ax}{\mathord{<}\x\mathord{>}}

\newcommand{\axy}{\mathord{<} x,y \mathord{>}}
\newcommand{\axxs}{\mathord{<} x,x^{\ss} \mathord{>}}

\makeatletter
\def\moverlay{\mathpalette\mov@rlay}
\def\mov@rlay#1#2{\leavevmode\vtop{%
    \baselineskip\z@skip \lineskiplimit-\maxdimen
    \ialign{\hfil$#1##$\hfil\cr#2\crcr}}}
\makeatother

\makeatletter
\def\Ddots{\mathinner{\mkern1mu\raise\p@
\vbox{\kern7\p@\hbox{.}}\mkern2mu
\raise4\p@\hbox{.}\mkern2mu\raise7\p@\hbox{.}\mkern1mu}}
\makeatother

\newcommand{\csff}[4]{-\langle{#1}^{\prime\prime}(#2)[#3]#4,#4\rangle}

\newcommand{\plangle}{\moverlay{(\cr<}}
\newcommand{\prangle}{\moverlay{)\cr>}}
\def\freef{\RR\plangle\x\prangle}
\def\freey{\RR\plangle x,y\prangle}

\newtheorem{theorem}{Theorem}
\newtheorem{lem}[theorem]{Lemma}
\newtheorem{thm}[theorem]{Theorem}
\newtheorem{cor}[theorem]{Corollary}
\newtheorem{proposition}[theorem]{Proposition}

\theoremstyle{definition}

\newtheorem{remark}[theorem]{Remark}
\newtheorem{example}[theorem]{Example}

\newtheorem{definition}[theorem]{Definition}

\newtheorem{exe}[theorem]{Exercise}

\def\beq{\begin{equation} }
\def\eeq{\end{equation} }

\def\bes{\begin{equation*} }
\def\ees{\end{equation*} }
\def\ben{\begin{enumerate}}
\def\een{\end{enumerate}}

\def\bexe{\begin{exe} }
\def\eexe{\end{exe} }

\def\bexa{\begin{example}  \rm }
\def\eexa{\end{example} }

\def\bmat{\begin{bmatrix}}
\def\emat{\end{bmatrix}}

\def\bev{\begin{verbatim}}
\def\eev{\end{verbatim}}

\def\ntn{n\times n}

\def\tX{{\tilde X}}

\def\nc{noncommutative}
\def\Nc{Noncommutative}

\def\eps{\epsilon}

\def\bpf{\begin{proof}}
\def\epf{\epf{proof}}

\def\bd{\partial \cD}

\def\qpos{QuadratischePositivstellensatz}

\def\dd{\delta}
\def\De{\Delta}
\def\Ga{\Gamma}

\newcommand{\df}[1]{{\it{#1}}{\index{#1}}}

{\index{#1}}

\newcommand\ncad[1]{{\tt #1}\index{#1}}
\newcommand\ncsosd[1]{{\tt #1}\index{#1}}

\makeindex

\title[Free Convexity]{Free Convex Algebraic Geometry}

\author[Helton]{J. William Helton${}^1$}
\address{J. William Helton\\ Department of Mathematics\\
  University of California \\
  San Diego}
\email{helton@math.ucsd.edu}
\thanks{${}^1$Research supported by NSF grants
DMS-0700758, DMS-0757212, and the Ford Motor Co.}

\author[Klep]{Igor Klep${}^2$}
\address{Igor Klep, Department of Mathematics\\ 
The University of Auckland\\ New Zealand}
\email{igor.klep@auckland.ac.nz}
\thanks{${}^2$Supported by the Faculty Research Development Fund (FRDF) of The
University of Auckland (project no. 3701119). Partially supported by the Slovenian Research Agency grant P1-0222.}

\author[McCullough]{Scott McCullough${}^3$}
\address{Scott McCullough\\ Department of Mathematics\\
  University of Florida, Gainesville 
   }
   \email{sam@math.ufl.edu}
\thanks{${}^3$Research supported by the NSF grants DMS-0758306 and DMS-1101137.}



\date{\today}

\begin{document}

\setcounter{tocdepth}{3}
\contentsmargin{2.55em} 
\dottedcontents{section}[3.8em]{}{2.3em}{.4pc} 
\dottedcontents{subsection}[6.1em]{}{3.2em}{.4pc}
\dottedcontents{subsubsection}[8.4em]{}{4.1em}{.4pc}

\begin{abstract}
 This chapter is a tutorial on techniques and results in
free convex algebraic geometry   and free real algebraic geometry (RAG).
   The term free  refers to the central role played
  by algebras of noncommutative polynomials $\RR\ax$
  in free (freely noncommuting) variables $x=(x_1,\dots,x_g)$.
The subject pertains to  problems where the unknowns are
matrices or Hilbert space operators as arise in
linear systems engineering and quantum information theory.

 The  subject of free RAG flows in  two branches.
 One,  free positivity and inequalities
 is an analog of classical real algebraic geometry,
 a theory of polynomial inequalities embodied in
 algebraic formulas called Positivstellens\"atze;
often free  Positivstellens\"atze have cleaner
  statements than their commutative counterparts.
 Free convexity, the second branch of free RAG,
 arose in an effort to unify a torrent of ad hoc
 optimization techniques which came on the linear
 systems engineering scene in the mid 1990's.
 Mathematically, much as in the
 commutative case, free convexity is connected with free positivity
 through the second derivative:
 A free polynomial is convex if and only if its Hessian is positive.
 However, free convexity is a very restrictive condition,
 for example, free convex polynomials have degree 2 or less.

This article describes for a beginner techniques involving
 free convexity.  As such it also serves as a point of entry into
 the larger field of  free real algebraic geometry. 
\end{abstract}

\maketitle

\section {Introduction}
  This chapter is a tutorial on techniques and results in
  \df{free convex algebraic geometry}
  and \df{free positivity}.   As such it 
  also serves as a point of entry into the larger field of 
  \df{free real algebraic geometry} (\df{free RAG}),
  and  makes contact
  with  noncommutative real algebraic geometry
 \cite{Hel02, HKM10c, HKM+, HKM++, HMppt, KlSch:08, KlSch:09a, 
McC01, PNA10,S05,SmuSurvey}, 
   free analysis\index{free analysis} and
  free probability\index{free probability} (lying at the origins of free analysis,
  cf.~\cite{AIM}),
  free analytic function theory and
 free harmonic analysis 
  \cite{HKM10a,HKM10b, HKMS09,MS+,Pop06,Voi04,Voi10, KVV-}.
  
   The term free here refers to the central role played 
  by algebras of noncommutative polynomials $\RR\ax$
  in free (freely noncommuting) variables $x=(x_1,\dots,x_g)$.  
  A striking difference between the free and classical
  settings is the following Positivstellensatz.

\begin{theorem}[Helton \cite{Hel02}]
 \label{thm:grandpos} 
   A nonnegative $($suitably defined$)$ free
  polynomial is a sum of squares. 
\end{theorem}\index{sum of squares}

 The  subject of free RAG flows in  two branches.
 One, 
 free positivity 
 is an analog of classical real algebraic geometry,
 a theory of polynomial inequalities embodied in Positivstellens\"atze.
 As is the case with the sum of
 squares result above (Theorem \ref{thm:grandpos}), 
 generally free  Positivstellens\"atze have cleaner 
  statements than do their commutative counterparts; 
see e.g.~\cite{McC01,Hel02, HMP04, HKM++} for a sample. 
 Free convexity, the second branch of free RAG,\index{free convexity}
 arose in an effort to unify 
 a torrent of ad hoc techniques which came on the linear 
 systems engineering scene in the mid 1990's. 
 We soon give a quick sketch of the engineering motivation,
 based on the slightly more complete sketch given in the 
  survey article \cite{OHMP09}. 
 Mathematically, much as in the
 commutative case, free convexity is connected with free positivity
 through the second derivative: A free polynomial is convex if and only if
 its Hessian is positive.

 The tutorial proper starts with Section \ref{sec:basics}. 
 In the remainder of this introduction, motivation for 
 the study of free positivity and convexity arising in 
 \df{linear systems engineering}, \df{quantum phenomena}, and
 other subjects such as \df{free probability} is provided,
 as are some suggestions for further reading.

\subsection{Motivation}\label{subsec:Motivation}
 While the theory is both mathematically pleasing and natural,
 much of the excitement of  
 free convexity and positivity stems from its applications.
 Indeed, the fact that a large class of 
  linear systems engineering problems naturally lead
  to free inequalities provided the main force behind the development of the
  subject.  In this motivational section, we describe in some detail
  the linear systems  point of view.  We also give a brief introduction
to other applications.

\subsubsection{Linear Systems Engineering}
The layout of a linear systems  problem 
is typically specified
by a signal flow diagram. Signals go into boxes and other signals
come out. 
The boxes in a linear system contain
constant coefficient  linear differential equations
which are specified entirely by matrices (the coefficients of the
  differential equations). 
Often many boxes appear and many signals
transmit between them.  
In a typical problem some boxes are 
given and some we get to design subject to the condition that 
the $L^2$ norm of various signals must compare 
in a prescribed way, e.g.~the input to the system has $L^2$
norm bigger than the output.
The signal flow diagram itself and corresponding problems
do not specify the size of matrices involved.
So ideally any algorithms derived apply to  matrices of all sizes.
Hence the problems are called \df{dimension free}.

An empirical observation is that system problems of this type
convert to inequalities on polynomials in matrices,
the form of the polynomials being determined entirely by the 
signal flow layout (and independent of the matrices involved).
Thus the systems problem naturally leads to
 free polynomials and free positivity conditions.

For yet a more detailed discussion of this example, see \cite[\S 4.1]{OHMP09}. 
Those who read Chapter 2 saw a basic example of this in Chapter 2.2.1. Next we give more
of an idea of how the correspondence between linear systems and noncommutative polynomials occurs. This is done primarily with an example.

\subsubsection{Linear systems}
A {\it linear system}\index{linear system} $\mathfrak F$ is given
by the constant coefficient linear differential equations
\begin{align*}
\frac{dx}{dt} &= Ax + Bu, \\
y &= Cx,
\end{align*}
with the vector
\begin{itemize}
\item
  $x(t)$ at each time $t$
being in the  vector space $\cX$ called the {\it state space},
\index{state space}
\item $u(t)$ at each time $t$ being in the
vector space $\cU$ called the {\it input space}, \index{input
space}
\item $y(t)$ at each time $t$ being in the vector space
$\cY$ called the {\it output space}\index{output space},
\end{itemize}
and $A,B,C$ being linear maps on the corresponding vector spaces.

\subsubsection{Connecting linear systems}
Systems can be connected in incredibly complicated configurations.
We  describe a simple connection and this goes a long way toward
illustrating the  general idea. Given two linear systems
$\mathfrak F$, $\mathfrak G$, we describe the formulas for
connecting them in feedback.

\def\fa{{Q}}
\def\fb{{R}}
\def\fc{{S}}

One basic feedback connection  is  described by the diagram  

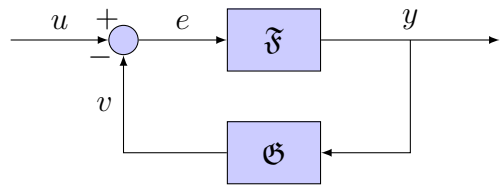
\begin{center}
\tikzstyle{block} = [draw, fill=blue!20, rectangle, 
    minimum height=2em, minimum width=3em]
\tikzstyle{sum} = [draw, fill=blue!20, circle, node distance=1.5cm]
\tikzstyle{input} = [coordinate]
\tikzstyle{output} = [coordinate]
\tikzstyle{pinstyle} = [pin edge={to-,thin,black}]
\begin{tikzpicture}[auto, node distance=2cm,>=latex]
    \node [input, name=input] {};
    \node [sum, right of=input] (sum) {};
    \node [block, right of=sum] (system) {$\mathfrak F$};
    \node [output, node distance=3cm, right of=system] (output) {};
    \node [block, node distance=1.5cm, below of=system] (measurements) {$\mathfrak G$};
    \draw [draw,->] (input) -- node {$u$} (sum) node [pos=0.99] {$+$};
    \draw [->] (sum) -- node {$e$} (system);
    \draw [->] (system) -- node [name=y] {$y$}(output);
    \draw [->] (y) |- (measurements);
    \draw [->] (measurements) -| node[pos=0.99] {$-$} 
        node [near end] {$v$} (sum);
\end{tikzpicture}
\end{center}
\noindent
called a {\it signal flow diagram}.\index{signal flow diagram}
Here $u$ is a signal going into the 
{\it closed loop system} \index{closed loop system}
and $y$ is the signal coming out.
The signal flow diagram is equivalent to a collection of 
equations.
The 
systems $\mathfrak F$ and $\mathfrak G$ 
themselves are respectively given by
the linear differential equations
\begin{align*}
\frac{dx}{dt} &= Ax+Be, &
\frac{d\xi}{dt} &= \fa \, \xi + \fb \, w, \\
y&=Cx, &  v&= \fc \, \xi.
\end{align*}
The feedback connection  is described
 algebraically by 
\begin{align*}
w &= y &
&\text{and} &
e&=u-v.
\end{align*}
Putting these relations together gives that 
the  closed loop system 
is described by differential equations 
\begin{align*}
  \frac{dx}{dt} &= Ax  - B \fc \xi + Bu, \\
  \frac{d\xi}{dt} &= \fa \, \xi+ \fb \, y = \fa \, \xi + \fb \, C x, \\
  y&=Cx.
\end{align*}
which is conveniently described in matrix form as 
\begin{equation}
  \label{eq:cloopsys}
  \begin{aligned}
    \frac{d}{dt}
    \begin{bmatrix} x\\ \xi \end{bmatrix} &= \begin{bmatrix} A & -B \fc \\ \fb\, C & \fa
    \end{bmatrix} \begin{bmatrix} x\\ \xi
    \end{bmatrix} + \begin{bmatrix} B\\0
    \end{bmatrix} u,  \\
    y &= \begin{bmatrix} C & 0 \end{bmatrix} \begin{bmatrix} x\\ \xi \end{bmatrix},
  \end{aligned}
\end{equation}
where the state space of the closed loop systems is the direct sum
$\cX \oplus \cY$ of the state spaces $\cX$ of $\mathfrak F$ and
$\cY$ of $\mathfrak G$.
From  \eqref{eq:cloopsys},  the coefficients of the O.D.E.
are (block) matrices whose entries are (in this case simple) 
polynomials in the matrices $A,B,C,Q,R,S$.

This illustrates the 
 moral of the general story:\\ 

{\it System connections produce a new system whose coefficients
are matrices with entries which are noncommutative polynomials
(or at worst ``rational expressions'')
in the coefficient matrices of the component systems. }\\

 Complicated signal
 flow diagrams give complicated matrices of noncommutative polynomials
or rationals.
 Note in
 what was said the dimensions of vector spaces and matrices 
 $A,B,C,Q,R,S$ 
 never
 entered explicitly; the algebraic form of (\ref{eq:cloopsys}) is
 completely determined by the flow diagram. Thus, such linear
 systems lead to  {\it dimension free}\index{dimension free} problems.

Next we turn to how ``noncommutative inequalities" arise.
The main constraint producing them can be thought of  as energy dissipation,
a special case of which are  the Lyapunov functions 
already seen in  Chapter 2.2.1.

\subsubsection{Energy dissipation}
\label{subsec:diss}

We have a system $\mathfrak F$ and want a condition which checks
whether
$$
\int_{0}^{\infty}{|u|}^{2}dt \geq \int_{0}^{\infty}{|\mathfrak
Fu|}^{2}dt,
    \qquad x(0)=0,
$$
holds for all input functions $u$, where $\mathfrak F u =y$ in the
above notation. If this holds $\mathfrak F$ is called a {\it
dissipative system}.\index{dissipative system}
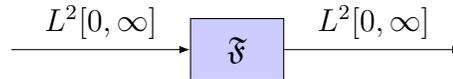
\begin{center}
\tikzstyle{block} = [draw, fill=blue!20, rectangle, 
    minimum height=2em, minimum width=3em]
\tikzstyle{sum} = [draw, fill=blue!20, circle, node distance=3cm]
\tikzstyle{input} = [coordinate]
\tikzstyle{output} = [coordinate]
\tikzstyle{pinstyle} = [pin edge={to-,thin,black}]
\begin{tikzpicture}[auto, node distance=3cm,>=latex]
    \node [input, name=input] {};
    \node [block, right of=input] (system) {$\mathfrak F$};
    \node [output, right of=system] (output) {$L^{2}[0, \infty]$};
    \draw [draw,->] (input) -- node [name=y] {$L^{2}[0, \infty]$}(system);
    \draw [->] (system) -- node [name=y] {$L^{2}[0, \infty]$}(output);
\end{tikzpicture}
\end{center}

 The energy dissipative condition is formulated
 in the language of analysis, but it converts to algebra
 (or at least an algebraic inequality) because of the
 following construction, which assumes the existence
 of a
 ``potential energy''-like
 function $V$ on the state space.
 A  function $V$ which satisfies  $V \geq 0, \ V(0)=0,$
 and
  $$
    V(x(t_{1})) + \int_{t_{1}}^{t_{2}} |u(t)|^{2} dt \ \ \geq \ \
    V(x(t_{2}))  + \int_{t_{1}}^{t_{2}}{|y(t)|^{2}} dt
   $$
for all input functions $u$ and initial states $x_1$ is
called a {\it storage function}\index{storage function}.
The displayed inequality is interpreted physically as \\

{\it
  potential energy now $+$ energy in $\geq$
potential energy then
$+$  energy out.}\\

 Assuming enough smoothness of $V$,
 we can differentiate this integral condition and
 use $\frac{d}{dt}x(t_{1}) = \ Ax(t_{1})+Bu(t_{1})$
   to obtain  a
 differential inequality
  \begin{equation}
    \label{eq:dissipux}
0 \ \geq \ \ \nabla  V(x) (Ax  + Bu) \ \ + \ \ |Cx|^{2} -
|u|^2,
  \end{equation}
on what is called the ``reachable set" (which we do not need to 
define here).
 
 In the case
of linear systems, $V$ can be chosen to be a quadratic.
So it has 
the form $V(x)=\langle Ex,x \rangle$ with $E \succeq 0$
and $\nabla V(x)=2Ex$.

\begin{thm}
    The linear system $A,B,C$
    is dissipative if inequality {\rm\eqref{eq:dissipux}}
    holds for all $u \in \cU, x \in \cX$.
    Conversely, if $A,B,C$ is ``reachable"\footnote{A mild technical condition},
    then dissipativity implies inequality {\rm\eqref{eq:dissipux}}
    holds for all $u \in \cU$, $x \in \cX$.
\end{thm}

 In the linear  case,
 we may substitute $\nabla V(x)= 2 Ex$ in (\ref{eq:dissipux}) to obtain
$$
 0\geq 2 (Ex)^\ss (Ax+Bu) + |Cx|^{2} - |u|^{2}, \ \
$$
for all $u,x$. 
Then maximize in $x$ to get
$$
 0 \geq x^\ss[EA + A^{\ss}E + E BB^{\ss} E + C^{\ss}C]x .
$$
Thus the classical {\it Riccati matrix inequality}
\index{Riccati matrix inequality}
\begin{equation}
 \label{eq:ricatti}
 0 \succeq \ EA + A^{\ss}E + EBB^{\ss}E + C^{\ss}C  \quad \text{with}\quad \quad E  \succeq 0
\end{equation}
ensures dissipativity of the system;  and, it turns out,  is also
implied by dissipativity when the system is reachable.

It is inequality  \eqref{eq:ricatti},
applied in many many contexts, which leads to positive semidefinite inequalities throughout all of linear systems theory.

As an aside we return to the very special case of 
dissipativity, namely   Lyapunov
stability, described in Chapter 2.2.1.
Our discussion starts with the ``miracle of inequality \eqref{eq:ricatti}":
when $B=0$ it becomes the Lyapunov inequality.
However, this is merely magic (no miracle whatsoever); the trick being that
the if input $u$ is identically zero, then dissipativity implies stability.
The converse is less intuitive, but true:  stability of $\dot{x}= Ax$
implies
existence of a ``virtual" potential energy $V(x)=\langle Ex,x \rangle$  and output $C$ making the ``virtual" system dissipative.

\subsubsection{Schur Complements and Linear Matrix Inequalities}
 \label{subsubsec:complement}

Using Schur complements,\index{Schur complement}
  the Riccati inequality of equation
(\ref{eq:ricatti}) is equivalent to the inequality
\bes
 L(E):=\begin{bmatrix} EA+A^{\ss}E+C^{\ss}C & EB\\
 B^{\ss}E & -I   \end{bmatrix} \preceq 0.
\ees
Here $A$, $B$, $C$
describe the system and $E$ is an unknown matrix. If the system is
reachable, then $A$, $B$, $C$ is dissipative if and only if $L(E)
\preceq 0$ and $E \succeq 0$.

The key feature in this reformulation of the Riccati
inequality is that $L(E)$ is linear in $E$, so
the inequality $L(E)\preceq 0$ is   a
\df{Linear Matrix Inequality} (LMI)\index{linear matrix inequality}\index{LMI}\index{inequality!linear matrix} in $E$.

\subsubsection{Putting it together}

We have shown two ingredients of linear system theory,
connection laws (algebraic) and dissipation (inequalities),
but have yet to put them together.
It is in fact a very mechanical procedure.
After going through the procedure one sees that the problem a
software toolbox designer  faces is this:
\begin{quote}
(GRAIL) \ \ 
Given a symmetric matrix of nc polynomials
$$p(a,x)= \Big[ p_{ij}(a,x) \Big]_{i,j=1}^k,$$
and a tuple of matrices $A,$  provide an
algorithm for finding $X$ making $p(A,X)\succeq0$ or
better yet as large as possible.
\end{quote}
Algorithms for doing this are based on numerical 
optimization or a close relative, so even if they find a local
solution there is  no guarantee that it is global.  
If $p$ is convex in $X$, then these problems disappear.

 Thus, systems problems described by
 signal flow diagrams produce 
a mess of matrix inequalities\index{matrix inequality}\index{inequality!matrix} with some matrices
known and some unknown and the constraints that
some polynomials are positive semidefinite.
The inequalities can get very complicated as one might 
guess, since signal flow diagrams get complicated.
These considerations thus naturally lead to 
the emerging subject of free real algebraic geometry, 
the study of noncommutative (free) polynomial inequalities
and free semialgebraic sets.  Indeed, much of what is known
about this very new subject is touched on in this chapter.

The engineer would like for these polynomial inequalities 
to be convex in the unknowns. Convexity guarantees that 
local optima are global optima 
(finding global optima is often of paramount importance)
and facilitates numerics.

Hence the major issues in linear systems theory are:
\ben[\rm(1)]
\item
  {\it  Which problems convert to a convex matrix inequality?
How does one do the  conversion?}
\item
 {\it Find numerics which will solve large convex problems.
How do you use special structure, such as most unknowns are matrices
and the formulas are all built of noncommutative rational functions?
}
\item
{\it
    Are convex matrix inequalities
    more general than LMIs? }
\een

The mathematics here  can be motivated
by the problem of writing a toolbox for engineers to use
in designing linear systems. 
What goes in such toolboxes is algebraic formulas
 with matrices $A,B, C$
unspecified and reliable numerics for solving them when a user
does specify $A,B, C$ as matrices.
A user who designs a controller
for a helicopter puts in the mathematical systems model for his
helicopter and puts in matrices, for example, $A$ is a particular
${8 \times 8}$ real matrix etc.
Another user who designs a satellite
controller might have a 50 dimensional state space and of course
would pick completely different $A,B,C$.
Essentially any matrices
of any compatible dimensions can occur.
Any claim we make about our formulas
 must be valid
regardless of the size of the matrices plugged in.

The toolbox designer faces two completely different tasks.
One is manipulation of algebraic inequalities;
 the other is numerical solutions.
 Often the first is far more daunting
 since the numerics is handled by some standard
 package (although for numerics problem size is a demon).
 Thus there is a great need for algebraic theory.
Most of this chapter bears on  questions like $(3)$ above
where
the unknowns are matrices.
{\it The first two questions will not be addressed.}
Here we treat (3) when there are no $a$ variables.
When there are $a$ variables see \cite{HHLM08,BMproc}.
Thus we shall consider  polynomials $p(x)$
in free noncommutative variables $x$ and focus on their convexity on 
free semialgebraic sets.

\smallskip

What are the implications of our study for engineering?
Herein you will see strong results on free convexity
but what do they say to an engineer?
We foreshadow the forthcoming answer by saying it is fairly negative,
but postpone further disclosure till the final page of these
writings not so much to promote suspense, but
for the conclusion to arrive  after you have absorbed the theory.

\subsubsection{Quantum Phenomena}
Free\index{Positivstellensatz} Positivstellens\"atze - algebraic certificates for positivity -
 of which Theorem \ref{thm:grandpos} is the grandad, have
 physical applications. 
Applications to quantum physics are explained
by Pironio, Navascu\'es, Ac\'\i n \cite{PNA10} 
who also consider computational aspects related to 
noncommutative sum of squares.\index{sum of squares}
How this pertains to operator algebras is discussed by 
Schweighofer and the second author in \cite{KlSch:08}.
The important 
Bessis-Moussa-Villani conjecture (BMV) from quantum statistical mechanics
is tackled in \cite{KlSch:09a,CaKlPo:10}.
Doherty, Liang, Toner, Wehner \cite{Wehn:08}
employ noncommutative positivity and the Positivstellensatz \cite{HM04b} 
of the first and the third author to consider
the quantum moment problem and multi-prover games.

A particularly elegant recent development, independent of the
line of history containing the work in this chapter, was initiated by
Effros. The classic ``perspective" transformation  carries a function
on $\RR^n$ to a function on $\RR^{n+1}$. It is used for various purposes,
one being in
algebraic geometry to produce ``blowups" of singularities
thereby  removing  them. It has the property that convex functions
map to convex functions. What about convex functions on free variables?
This question was asked by Effros and settled affirmatively
in \cite{E09}
for natural cases as a way to show that quantum relative
entropy is convex.
Subsequently, \cite{ENGprept} showed that the perspective
transformation  in free variables
always maps convex functions to convex functions.

\subsubsection{Miscellaneous  applications}

A number of other scientific disciplines use free analysis, though less
  systematically than in free real algebraic geometry. 

\noindent
{\bf Free probability.}
\index{free probability}
Voiculescu developed it to attack  one of
the purest of mathematical questions regarding 
 von Neumann algebras.
From the outset (about 20 years ago)
 it was elegant and it came to have great
depth. Subsequently, it was discovered to bear
forcefully and effectively on random matrices.
The area is vast, so we do not dive in
but refer the reader to an introduction \cite{AIM,VDN}.

\noindent
{\bf Nonlinear engineering systems.} \
A classical technique in nonlinear  systems theory 
developed by Fliess is based on manipulation
of power series with noncommutative variables
(the Chen series).  The area
has a new impetus coming  from the problem of data compression,
so now is a time when these correspondences are being worked out,
cf. \cite{GL,GT,LCL}.

\subsection{Further reading}
 We pause here to offer some suggestions for further reading.
  For further engineering motivation we recommend the paper
  \cite{SI95} or the longer version
  \cite{SIG97}
  for related new directions.
 Descriptions of Positivstellens\"atze\index{Positivstellensatz} are in the surveys  
 \cite{HKManjos, OHMP09, HP07,SmuSurvey}
 with the first three also briskly touring free convexity.\index{free convexity}
 The survey article \cite{HMPV} is aimed at engineers. 

  Noncommutative is a broad term, encompassing
  essentially all algebras.  
  In between the extremes
  of commutative and free lie many important topics, 
  such as Lie algebras, Hopf algebras, quantum groups, 
$C^*$-algebras, von Neumann algebras, etc.
  For instance, there are elegant
  noncommutative real algebraic geometry
  results for the  Weyl Algebra \cite{S05},   
cf.~\cite{SmuSurvey}.

\subsection{Guide to the chapter}
 The goal of this tutorial is to introduce
  the reader to the main results and techniques used
  to study free convexity.  Fortunately, the subject is 
 new and the techniques not too numerous
  so that one can quickly become an expert. 

 The basics of free, or nc, polynomials and their evaluations
  are developed
  in Section \ref{sec:basics}. The key notions are 
  positivity and convexity for free polynomials. 
  The principal fact is that the second directional derivative
 (in direction $h$) of a free convex
  polynomial is a  positive quadratic polynomial in $h$
 (just like in the commutative case).
 Free quadratic (in $h$) polynomials  have a Gram type representation
 which thus figures prominently in studying convexity.  
 The nuts and bolts of this Gram representation and 
  some of its consequences, including Theorem
  \ref{thm:grandpos},  are the subjects of
  Sections \ref{sec:middle-border} and \ref{sec:qposSS} respectively.

 The Gram representation techniques actually require only 
 a small amount of convexity and thus there is a theory of 
 geometry on free varieties having signed (e.g.~positive) curvature.
 Some details are in Section \ref{sec:variety}.

 A couple of free real algebraic geometry results which have 
 a heavy convexity component are described in the last section,
 Section \ref{sec:convexsemialg}.
 The first is an optimal free convex Positivstellensatz which
 generalizes Theorem \ref{thm:grandpos}.  The second says that
 free convex semialgebraic sets are free
 spectrahedra, giving another example of the much more rigid
 structure in the free setting.

 Section \ref{sec:compAlg} 
 introduces software which handles free noncommutative computations.
 You may find it useful 
  in your free studies.

\smallskip
 In what follows, mildly incorrectly, but in keeping with
 the usage in the literature,
 the terms  noncommutative (abbreviated nc) and free 
 are used synonymously.

\section{Basics of nc Polynomials and their Convexity}
 \label{sec:basics}

 This section treats the  basics of polynomials in 
 nc variables, nc differential 
 calculus, and nc inequalities. There is
 also a brief introduction to nc rational functions and 
 inequalities.

\subsection{\Nc\   polynomials}
 \label{subsec:NCpoly}
 Before turning to the formalities,
 we give, by examples,  an informal introduction
 to noncommutative (nc) polynomials.
\index{polynomial}\index{nc polynomial}\index{noncommutative polynomial}\index{polynomial!nc}\index{polynomial!noncommutative}

 A noncommutative polynomial $p$ is a polynomial in
 a finite set $x=(x_1,\dots,x_g)$ of relation free variables.  A canonical
 example, in the case of two variables $x=(x_1,x_2),$ is
the \df{commutator}
\beq\label{eq:commutator}
  c(x_1,x_2)=x_1x_2-x_2x_1.
\eeq 
 It is precisely the fact that $x_1$ and $x_2$ do not commute that
 makes $c$ nonzero. 

 While a commutative polynomial $q\in \RR[t_1,t_2]$ is
 naturally evaluated at points $t \in \RR^2$, nc polynomials
 are naturally evaluated on tuples of square matrices.  For
 instance, with
\[
  X_1=\begin{bmatrix} 0 & 1 \\ 1 & 0 \end{bmatrix}, \ \ 
  X_2=\begin{bmatrix} 1 & 0 \\ 0 & 0 \end{bmatrix},
\]
  and $X=(X_1,X_2)$, 
 one finds
\[
  c(X) = \begin{bmatrix} {\color{white}-}0 & 1 \\ -1 & 0 \end{bmatrix}.
\]

  Importantly, $c$ can be evaluated on any pair $(X,Y)$ of
  symmetric matrices of the same size. (Later in the section  we will also 
  consider evaluations involving not necessarily symmetric
  matrices.)  Note that 
  if $X$ and $Y$ are $n\times n$, then $c(X,Y)$
  is itself an $n\times n$ matrix.  In the case
  of $c(x,y)=xy-yx$, the matrix
  $c(X,Y)=0$ if and only if $X$ and $Y$ commute. In
  particular, $c$ is zero on $\RR^2$ ($2$-tuples of $1\times 1$
  matrices). 

For another example, if $d(x_1,x_2)=1+x_1 x_2 x_1$, then with $X_1$ and
$X_2$ as above, we find
\[
d(X) = I_2 + X_1X_2X_1 = \begin{bmatrix}
1 & 0 \\ 0 & 2
\end{bmatrix}.
\]

  Note that although  $X$ is a tuple of symmetric matrices, 
  it need not be the case that $p(X)$ is symmetric. 
  Indeed, the matrix $c(X)$ above is not. 
  In the present context, 
  we say that $p$ is {\it symmetric},
  \index{symmetric polynomial}\index{polynomial!symmetric} if $p(X)$ is symmetric
  whenever $X=(X_1,\dots,X_g)$ is a tuple of symmetric
  matrices.  Another more algebraic  definition  
  of symmetric for nc polynomials
  appears in  Section \ref{sec:gitmo}.

\subsubsection{Noncommutative convexity for polynomials}
  Many standard notions for polynomials, and even functions,
  on $\RR^g$  extend
  to the nc setting, though often with unexpected ramifications.  For example,
  the commutative polynomial $q\in\RR[t_1,t_2]$ is convex if,
  given $s,t\in\RR^2$,
\[
    \frac12 \big(q(s)+q(t)\big) \ge q\Big(\frac{s+t}{2}\Big).
\]

  There is a natural ordering on symmetric $n\times n$ 
  matrices defined by $X\succeq Y$ if the symmetric
  matrix $X-Y$ is positive semidefinite; i.e., if its
  eigenvalues are all nonnegative.  
  Similarly, $X\succ Y$, if $X-Y$ is positive definite; 
  i.e., all its eigenvalues are positive.  This
  order yields a canonical notion of
  convex nc polynomial. 
  Namely, a symmetric polynomial $p$ is {\it convex} 
  \index{convex polynomial}\index{polynomial!convex}
  if for each $n$ and each pair of $g$ tuples of $n\times n$
  symmetric matrices $X=(X_1,\dots,X_g)$ and $Y=(Y_1,\dots,Y_g)$,
 we have
\[
    \frac12 \big(p(X)+p(Y)\big) \succeq p\Big(\frac{X+Y}{2}\Big). 
\]
  Equivalently,
\begin{equation}
 \label{eq:convexp}
    \frac{p(X)+p(Y)}2 - p\Big(\frac{X+Y}{2}\Big)  \succeq 0.
\end{equation}

 Even in one variable, convexity for an nc polynomial
 is a serious constraint.  For instance, consider 
 the polynomial $x^4$.  It is symmetric, but with 
$$
 X=\begin{bmatrix} 4&2\\2&2\end{bmatrix}
 {\rm  and }\; 
 Y=\begin{bmatrix} 2&0\\0&0\end{bmatrix}
$$
 it follows that 
$$
\frac{X^4+ Y^4}2- \Big(\frac12 X+\frac12 Y\Big)^4 =
\begin{bmatrix} 164&120\\120& {\color{white}0}84\end{bmatrix}
$$
 is not positive semidefinite. Thus $x^4$ is not convex.
 
\subsubsection{Noncommutative polynomial inequalities and convexity}\label{subsec:ncsa}
 The study of polynomial inequalities, real algebraic geometry or
 semialgebraic geometry, has a nc version.  A 
 \df{basic open  semialgebraic set} is a subset of $\RR^g$
  defined by  a list of polynomial inequalities; i.e.,
 a set $S$ is a basic open semialgebraic set\index{semialgebraic set} if
\[
  S=\{ t\in\RR^g \colon p_1(t)>0,\ldots,p_k(t)> 0\}
\] 
  for some polynomials $p_1,\dots,p_k\in\RR[t_1,\dots,t_g]$.

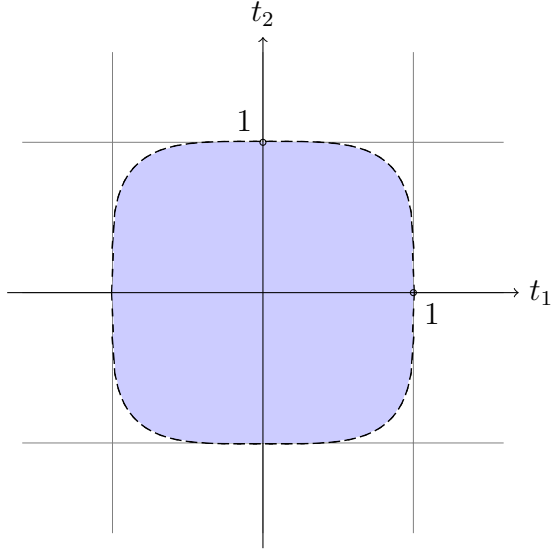
\def\nos{100}
\begin{center}
\begin{tikzpicture}[domain=-2:2,scale=2] 
\draw[very thin,color=gray] (-1.6,-1.6) grid (1.6,1.6);
\draw[color=black, domain=1:-1, very thick, dashed, samples=\nos] plot (\x,sqrt{(sqrt{(1-(\x)^4)})});
\draw[color=black, domain=1:-1, very thick, dashed, samples=\nos] plot (\x,-sqrt{(sqrt{(1-(\x)^4)})});
\draw[color=black, domain=-1:1, very thick, dashed, samples=\nos] plot (\x,sqrt{(sqrt{(1-(\x)^4)})});
\draw[color=black, domain=-1:1, very thick, dashed, samples=\nos] plot (\x,-sqrt{(sqrt{(1-(\x)^4)})});
 \fill[color=blue!20, domain=1:-1, samples=\nos] plot (\x,sqrt{(sqrt{(1-(\x)^4)})}); 
 \fill[color=blue!20, domain=1:-1, samples=\nos] plot (\x,-sqrt{(sqrt{(1-(\x)^4)})});
 \fill[color=blue!20, domain=-1:1, samples=\nos] plot (\x,sqrt{(sqrt{(1-(\x)^4)})}); 
 \fill[color=blue!20, domain=-1:1, samples=\nos] plot (\x,-sqrt{(sqrt{(1-(\x)^4)})});
\fill[blue!20] (0,0) circle (1);
\draw[->] (-1.7,0) -- (1.7,0) node[right] {$t_1$}; \draw[->] (0,-1.7) -- (0,1.7) node[above] {$t_2$};
\draw[color=black] (1,0) circle (0.02) node[below right] {$1$};
\draw[color=black] (0,1) circle (0.02) node[above left] {$1$};
\end{tikzpicture} 
\[
  {\rm ncTV}(1) = \{(t_1,t_2)\in\RR^2 \colon 1-t_1^4-t_2^4 > 0\}.
  \]
\end{center} 

  Because noncommutative polynomials are evaluated
 on tuples of matrices, 
  a nc (free) basic open semialgebraic set is a 
  sequence. 
  For positive integers $n$, let $\gtupn$
  denote the set of $g$-tuples of $n\times n$ symmetric matrices.
  Given symmetric nc polynomials $p_1,\dots,p_k$, let
\[
   \cP(n)=\{ X\in\gtupn \colon p_1(X)\succ 0, \ldots, p_k(X)\succ0\}.
\]
  The sequence $\cP=(\cP(n))$ is then a 
  {\it nc $($free$)$ basic open  semialgebraic set}.\index{free semialgebraic set}\index{semialgebraic set!free}\index{semialgebraic set!noncommutative}\index{semialgebraic set!nc}\index{nc basic open semialgebraic set}
   The sequence
\[
  {\rm ncTV}(n) =\{X\in\twotupn\colon I_n - X_1^4 - X_2^4 \succ 0\}
\]
 is an entertaining example. 
 When $n=1,$ ncTV$(1)$ 
  is a subset of $\RR^2$ often called the \df{TV screen}. 
  Numerically it can be verified, though it rather
  tricky to do so  
  (see Exercise \ref{exe2:eval}) 
  that the set ${\rm ncTV}(2)$ is not  a convex set. 
  An analytic proof
  that ${\rm ncTV}(n)$ is not a  convex set for some $n$ 
  can be found in \cite{DHM07a}.  It also follows by combining
  results in \cite{HMppt} and \cite{HV07}. 
For properties of the classical commutative
TV screen, see  the Chapters  6of Nie and 
5 by Rostalski-Sturmfels in this book. 

\bexa
Let $p_\eps:= \eps^2-\sum_{j=1}^g x_j^2$.  Then 
the \df{$\eps$-neighborhood
of $0$},
\bes
\cN_\eps:= \bigcup_{n\in\NN} \{ X\in \gtupn \colon p_\eps(X)\succ0\}
\ees
is an important example of a nc basic open semialgebraic set.
\eexa

\subsection{\Nc \ polynomials, the formalities}
 \label{sec:gitmo}
  We now take up the formalities of nc polynomials,
  their evaluations, convexity, and positivity.
  
Let $x=(x_{1},\ldots,x_{g})$ denote a $g$-tuple
 of free noncommuting variables
 and let  $\RR\ax$ denote the associative
$\RR$-algebra freely generated by $\x$, i.e., the elements of $\RR\ax$
are polynomials in the noncommuting variables $\x$ with coefficients
in $\RR$. Its elements are called {\it $($nc$)$ polynomials}.
 \index{nc polynomial}\index{polynomial}\index{polynomial!nc}
An element of the form $aw$ where $0\neq a\in \RR$ and
$w$ is a \df{word} in the variables $x$
  is called a \df{monomial} and $a$ its
 \df{coefficient}.
  Hence words are monomials whose coefficient is
$1$.  Note that the empty word $\emptyset$ plays the role
 of the multiplicative identity for $\RR\ax$.

 There is a natural \df{involution} ${}^\ss$  on $\RR\ax$ 
 that reverses words. For example, 
$(2-  3 x_{1}^2 x_{2} x_{3})^\ss =2  -3 x_{3} x_{2} x_{1}^2.$
  A polynomial $p$ is a \df{symmetric polynomial} 
 \index{polynomial!symmetric} if $p^\ss=p$.  
 Later we will see that this notion of symmetric is equivalent
 to that in the previous subsection. For now we note that
 of
\[
 \begin{split}
  c(x)&=  x_1x_2  - x_2 x_1 \\
  j(x)&=  x_1 x_2 + x_2 x_1 \\
 \end{split}
\]
  $j$ is symmetric, but $c$ is not. Indeed, $c^\ss=-c$.
  Because $x_j^{\ss}=x_j$ we refer to the variables 
  as \df{symmetric variables}.
  Occasionally we emphasize this point by writing
 $\RR\mathord{<} x=x^{\ss} \mathord{>}$  for $\RR\ax$.

 The \df{degree}  of an nc polynomial $p$, denoted $\deg(p)$, is 
 the length of the longest word appearing in $p$.  For  instance 
 the polynomials $c$ and $j$ above both have degree two and the
 degree of
\[
  r(x)=1 - 3x_1x_2-3x_2x_1 - 2 x_1^2x_2^4x_1^2 
\]
  is eight. 
 Let $\RR\ax_k$ \index{$\RR\ax_k$} denote the 
 polynomials of degree at most $k$.
 \index{degree!at most $k$}

\subsubsection{\Nc \ matrix polynomials}
 Given positive integers $d,d^\prime\in\NN$, let $\Mdd\ax$ denote the
  $d\times d^\prime$ matrices with entries from $\RR\ax$.  
  Thus elements of $\Mdd\ax$ are \df{matrix-valued nc polynomials}.
  The involution on $\RR\ax$ naturally extends to 
 a mapping $\ss:\Mdd\ax\to \RR^{d'\times d}\ax$. 
  In particular, if
\[
   P = \begin{bmatrix} p_{i,j} \end{bmatrix}_{i,j=1}^{d,d^\prime} 
       \in \Mdd\ax,
\]
  then
\[
  P^\ss = \begin{bmatrix} p_{j,i}^\ss \end{bmatrix}_{i,j=1}^{d,d^\prime} 
       \in \RR^{d^\prime\times d}\ax.
\]
  In the case that $d=d^\prime$, such a $P$ is symmetric if $P^\ss=P$.
  
\subsubsection{Linear pencils}
Given a positive integer $n$, let 
  $\Snn$ denote the real  symmetric $n\times n$ matrices.
   For $A_0,A_1,\dots,A_{g} \in \SRdd$,  the expression
\beq\label{eq:pencil}
  L(x)=A_0+\sum_{j=1}^{g} A_j x_j \in \SRdd\ax
\eeq
  in the noncommuting variables $x$
  is a \df{symmetric  affine linear pencil}.
 \index{affine linear pencil}\index{symmetric affine linear pencil}
\index{linear pencil}\index{linear pencil!affine}\index{linear pencil!symmetric}
 In other words, these are precisely the
 symmetric degree one matrix-valued nc polynomials.
If $A_0=I$, then $L$ is \df{monic}.\index{linear pencil!monic}
If $A_0=0$, then $L$ is a \df{linear pencil}. 
 \index{linear pencil} \index{homogeneous linear pencil}\index{linear pencil!homogeneous}
The homogeneous linear part $\sum_{j=1}^{g} A_j x_j$ of a linear
pencil $L$ as in \eqref{eq:pencil} will be denoted by $L^{(1)}$.

\begin{example}\rm
  Let
\[
  A_1 = \begin{bmatrix} 0 & 1 & 0 & 0 \\ 1 & 0 & 0 &0 \\
            0 & 0 & 0 & 0\\ 0 & 0 &0 &0 \end{bmatrix}, \quad
  A_2 = \begin{bmatrix} 0 & 0 & 0 &0\\ 0 & 0 & 1 & 0\\
            0 & 1&0&0\\ 0&0&0&0\end{bmatrix}, \quad
  A_3 =\begin{bmatrix} 0&0&0&0\\ 0&0&0&0\\ 0&0&0&1\\0&0&1&0\end{bmatrix}.
\]
  Then
\[
  I+\sum A_j x_j 
   =\begin{bmatrix} 1 & x_1 &0&0\\x_1&1&x_2&0\\0&x_2&1&x_3\\0&0&x_3&1
  \end{bmatrix}
\]
is the corresponding monic affine linear pencil.
\end{example}

\subsubsection{Polynomial evaluations}
 If $p\in\Mdd\ax$ is an nc polynomial and
  $X\in\gtupn$, 
 the \df{evaluation}\index{evaluation!polynomial}\index{polynomial!evaluation} $p(X)\in\RR^{dn\times d'n}$ is defined by 
 simply replacing $x_{i}$ by $X_{i}$. 
  Throughout we use lower case letters for variables and 
  the corresponding capital 
   letter for matrices 
  substituted for that variable.

\bexa
Suppose 
$
  p(x)=A x_1 x_2   
$
 where
$
A=\begin{bmatrix}
-4&2\\
{\color{white}-}3&0
\end{bmatrix}. 
$
That is,
$$
p(x)= \begin{bmatrix}  -4 x_1 x_2 & 2 x_1 x_2 \\ {\color{white}-}3x_1 x_2 & 0 \end{bmatrix}.    
$$
Thus $p \in \RR^{2 \times 2}\ax$ and 
one example of an  evaluation is
$$
p\left( \begin{bmatrix}
0&1 \\
1 & 0
\end{bmatrix},
\begin{bmatrix}
1&{\color{white}-}0\\
0&-1
\end{bmatrix}
\right)=
A\otimes \left( \begin{bmatrix}
0&1 \\
1 & 0
\end{bmatrix}
\,
\begin{bmatrix}
1&{\color{white}-}0\\
0&-1
\end{bmatrix}
\right)=
A\otimes \left( \begin{bmatrix}
0&-1 \\
1 & {\color{white}-}0
\end{bmatrix}
\right)
$$
$$
=
\begin{bmatrix}
 {\color{white}-}0 & {\color{white}-}4 &  0 & -2 \\
 -4 & {\color{white}-}0 &  2 & {\color{white}-}0 \\
 {\color{white}-}0 & -3 & 0 & {\color{white}-}0 \\
 {\color{white}-}3 & {\color{white}-}0 &  0 & {\color{white}-}0 \\
\end{bmatrix}.
$$
 
Similarly, if $p$ is 
 a constant matrix-valued nc polynomial,
  $p(x)=A,$ and $X\in\gtupn$, 
then $p(X)=A\otimes I_n$.
Here we have taken advantage of the usual
 {\em tensor} (or {\em Kronecker}) {\em product}\index{tensor product}\index{tensor product!Kronecker} of matrices. 
  Given an $\ell\times \ell^\prime$ matrix  $A=(A_{i,j})$ and
  an $n\times n^\prime$ matrix  $B$, by definition, 
  $A\otimes B$ is the $n\times n^\prime$ {\it block} matrix
\[
  A\otimes B = \begin{bmatrix} A_{i,j}B \end{bmatrix},
\]
  with $\ell \times \ell^\prime$ matrix entries.
  We have reserved the tensor product notation for the tensor product
of matrices and have eschewed the strong temptation of using 
  $A\otimes x_{\ell}$ in place of
$Ax_{\ell}$ when $x_{\ell}$ is one of the variables.
\eexa

\begin{proposition}
 \label{prop:faith}
  Suppose $p\in\RR\ax.$   
  In increasing levels of generality,
\ben[\rm(1)]
 \item  if $p(X)=0$ for all $n$ and all $X\in (\Snn)^g$, 
  then $p=0$;
 \item  if there is a nonempty nc basic open semialgebraic
  set $\cO$ such that $p(X)=0$ on $\cO$ 
  $($meaning for every $n$ and $X\in\cO(n)$, $p(X)=0)$, then $p=0$;
 \item there is an $N,$ depending only
  upon the degree of $p$,  so that  for any $n \ge N$ 
  if there is an open subset $O\subseteq \gtupn$ with $p(X)=0$
   for all $X\in O$, then $p=0$.
\een
\end{proposition}\index{polynomial identity}\index{polynomial!vanishing}

\begin{proof}
  See Exercises \ref{exe:fock}, \ref{exe:openfock},
  and \ref{exe:openinadim}. 
\end{proof}

\bexe
 Use Proposition {\rm\ref{prop:faith}} to prove the following statement:

\begin{proposition}
 \label{prop:symmetric}
  Suppose $p\in \RR\ax$.  Show $p(X)$ is symmetric for every $n$
  and every $X \in (\Snn)^g$ if and only if $p^\ss=p$.
\end{proposition}
\eexe

\label{symmetric polynomial}\index{polynomial!symmetric}

\subsection{\Nc \ convexity revisited and nc positivity}
Now we return with a bit more detail on our main theme, 
convexity.
 A 
 symmetric polynomial $p$ is \df{matrix convex},
 \index{convex polynomial}\index{polynomial!convex}
if for each
positive integer $n$, each pair of
$g$-tuples $X=(X_1,\dots,X_g)$ and $Y=(Y_1,\dots,Y_g)$ in
$(\SRnn)^g$  and each $0\le t \le 1$,
\bes
  tp(X)+(1-t)p(Y) - p\big(tX+(1-t)Y\big)\succeq 0,
\ees
where, for an $n\times n$ matrix $A\in\mathbb{R}^{n\times n}$, the notation
$A\succeq 0$ means $A$ is positive semidefinite.  Synonyms for matrix
 convex include both 
 \df{nc convex}, and simply \df{convex}.
 \index{convex} \index{nc convex}

\bexe
  Show that the definition here of $($matrix$)$ convex is equivalent to that given in
  equation \eqref{eq:convexp} in the informal introduction to nc polynomials.
\eexe

 As we have already seen in the informal introduction to
 nc polynomials, 
even in one-variable, convexity in the noncommutative setting
differs from convexity in the commutative case because here $Y$ need
not commute with $X$.
Thus, although the polynomial $x^4$ is a convex function of one real
variable, it is not matrix convex.
On the other hand,  to verify that $x^2$ is a matrix convex polynomial,
observe that
\begin{eqnarray}
\nonumber
tX^2+(1-t)Y^2&-&(tX+(1-t)Y)^2\\
\nonumber &=&t(1-t)(X^2-XY-YX+Y^2)=t(1-t)(X-Y)^2 \succeq 0.
\end{eqnarray}

A polynomial $p\in\RR\ax$ is \df{matrix
positive}, synonymously \df{nc positive} or simply \df{positive}
 \index{matrix positive}\index{positive polynomial}\index{polynomial!positive}
  if $p(X)\succeq 0$ for all tuples $X=(X_1,\dots,X_g)\in
(\SRnn)^g$.  A polynomial $p$ is a \df{sum of squares}
 if there exists $k\in\NN$ and polynomials
  $h_1,\ldots, h_k$ such that
\bes
  p=\sum_{j=1}^k h_j^{\ss} h_j.
\ees
  Because, for a matrix $A$, the matrix $A^{\Ms}A$ is positive semidefinite, 
 if $p$ is a sum of squares, then $p$ is positive. 
  Though we will not discuss its proof in this chapter, we mention
 that, in contrast with the commutative case, the converse is true
\cite{Hel02,McC01}.

\begin{theorem}\label{thm:ncsos}
  If $p\in\RR\ax$ is positive, then $p$ is a sum of squares. 
\end{theorem}

 As for convexity, note that $p(x)$ is convex if and only if the
 polynomial $q(x,y)$ in $2g$ nc variables given by
\[
 q(x,y)= \frac12 \big( p(x)+p(y)\big)-p\Big(\frac{x+y}{2}\Big)
\]
 is positive.\index{convex polynomial}\index{polynomial!convex}

\subsection{Directional derivatives vs.~nc convexity and positivity}
 Matrix convexity can be formulated in terms of positivity
  of the Hessian, just as in the case of a real variable. 
 Thus we take a few moments to develop a very useful 
 nc calculus.

Given
a polynomial $p\in\RR\ax$, 
the $\ell^{\rm th}$ \df{directional derivative} of $p$ 
 in the ``direction'' $h$ is
\bes
p^{(\ell)}(x)[h]:=\left.\frac{d^\ell p(x+ t h)}{d t^\ell}\right|_{t=0}.
\ees
 Thus $p^{(\ell)}(x)[h]$ is the polynomial 
that evaluates to 
$$
\left.
\frac{d^\ell p(X+tH)}{dt^\ell}\right|_{t=0}\quad\textrm{for every choice of}\quad 
X,\,H\in\gtupn\,.
$$
We let
  $p^\prime(x)[h]$ denote the first derivative and 
 the \df{Hessian}, 
 denoted $ p^{\prime\prime}(x)[h]$
of $p(x)$, is the second directional derivative of $p$ in 
  the direction $h$.    
  
Equivalently, the Hessian of $p(x)$ can also be defined as the part of the
polynomial
$$r(x)[h]:=2\big(p(x+h)-p(x)\big)$$
in
$$\RR\ax[h]:=\RR \mathop{<} x_1,\dots,x_g,\, h_1,\dots,h_g\mathop{>}$$
that is homogeneous of degree two in $h$.

If
$p^{\prime\prime} \neq 0$, that  is, if
$p=p(x)$ is an nc polynomial of degree two or more, then
the polynomial  $p^{\prime\prime}(x)[h]$ in 
the $2g$ variables $x_1,\ldots,x_g,h_1\ldots,h_g$
is homogeneous of degree two in $h$ and
has degree equal to the degree of $p$.

\bexa 
\mbox{}\par
\ben[\rm(1)]
\item
 The Hessian of the polynomial $p=x_1^2 x_2$ is
$$
p^{\prime\prime}(x)[h]= 2(h_1^2 x_2 + h_1 x_1 h_2 + x_1 h_1 h_2).
$$
\item
 The Hessian of the polynomial $f(x)=x^4$ (just one variable) is
\[
  f^{\prime\prime}(x)[h] =2(h^2 x^2 + hxhx+ hx^2h + xhxh+xh^2x+x^2h^2).
\]
\een
\eexa

NC convexity is neatly described in terms of the Hessian. 

\begin{lem}\label{lem:CP}
$p\in\RR\ax$ is nc convex if and only if $p''(x)[h]$ is nc positive.
\end{lem}\index{convex polynomial}\index{polynomial!convex}\index{polynomial!positive}

\begin{proof}
See Exercise \ref{exe:geomHess}.
\end{proof}

\subsection{Symmetric, free, mixed, and classes of  variables}
  To this point, our variables $x$ have been 
  {\it symmetric}
  \index{variables!symmetric} in the sense that,
  under the involution, $x_j^{\ss}=x_j$.
  The corresponding  polynomials, elements of
  $\RR\ax$  are then the nc 
  analog of polynomials in real variables, with
  evaluations at tuples in $\Snn$.  In 
  various applications and settings it is natural
  to consider nc polynomials in other types of variables.

\subsubsection{Free variables}
  The
  nc analog of polynomials in complex variables
  is obtained by allowing evaluations on 
  tuples $X$ of not necessarily symmetric matrices.
  In this case, the involution must be interpreted
  differently and the variables are 
  called {\it free}.\index{variables!free}

  In this setting, given the nc variables
 $x=(x_1,\dots,x_g)$, let $x^{\ss}=(x_1^{\ss},\dots,x_g^{\ss})$
  denote another collection of nc variables. 
  On the ring $\RR\axxs$ 
  \index{$\RR\axxs$} define the involution ${}^{\ss}$
   by the requiring 
   $x_j\mapsto x_j^{\ss}$; $x_j^{\ss} \mapsto x_j$; 
  ${}^{\ss}$ reverses the order of words; and linearity.   For instance,
  for
\[
  q(x)= 1+x_1^{\ss}x_2-x_2^{\ss}x_1\in\RR\axxs,
\]
 we have
\[
  q^{\ss}(x)=1+x_2^{\ss} x_1 -x_1^{\ss} x_2.
\]
  Elements of $\RR\axxs$ are polynomials in \df{free variables}
 \index{free variables} and in this setting the
  variables themselves are \df{free}. 

  A polynomial $p\in\RR\axxs$ is {\it symmetric}\index{symmetric polynomial}\index{polynomial!symmetric} provided
  $p^{\ss}=p$. In particular, $q$ above is not symmetric,  but
\begin{equation}
 \label{eq:free-exa}
  p= 1+x_1^{\ss}x_2 + x_2^{\ss}x_1
\end{equation}
 is. 
 
  A polynomial $p\in\RR\axxs$ is {\it analytic}\index{analytic polynomial}\index{polynomial!analytic}
 if
  there are no transposes; i.e., if $p$ is a polynomial
  in $x$ alone.

  Elements of $\RR\axxs$ are naturally evaluated on tuples
  $X=(X_1,\dots,X_g) \in(\Mnns)^g$. For instance, 
  if $p$ is the polynomial in equation \eqref{eq:free-exa}
  and $X=(X_1,X_2)\in (\RR^{2\times 2})^2$ where
\[
  X_1=\begin{bmatrix} 0 & 0 \\ 1 & 0 \end{bmatrix} = X_2 
\]
  then
\[
 p(X)=\begin{bmatrix} 3 & 0 \\ 0 & 1 \end{bmatrix}.
\]

 The space $\Mdd\axxs$ is defined by analogy
 with $\Mdd\ax$ and evaluation of elements in $\Mdd\axxs$
 at a tuple $X\in (\Mnns)^g$ is defined in the obvious way.

\bexe
  State and prove analogs of Propositions {\rm\ref{prop:faith}} 
   and {\rm\ref{prop:symmetric}}
  for $\RR\axxs$ and evaluations from  $(\Mnns)^g$.   
\eexe

\subsubsection{Mixed variables}
  At times it is desirable to mix free and symmetric variables.
  We won't introduce notation for this situation as it will
  generally be understood from the context.  Here
  are some examples:

\bexa
 \begin{equation}
\label{2b}
p(x) = x_1^\ss x_1 +  x_2 + \frac{3}{4}  x_1  x_2x_1^\ss,
\ \ \ \ x_2=x_2^\ss; 
\end{equation}
\bes
{\rm ric}(a_1,a_2,x) = a_1 x + x a_1^\ss - x a_2 a_2^\ss x, \ \ \ \ x=x^\ss,
\ees
 In the first case $x_1$ is free, but $x_2$ is symmetric; and
 in the second $a_1$ and $a_2$ are free, but $x$ is symmetric.
 Two additional remarks are in order about the 
 second polynomial.  First, 
 it  is a \df{Riccati polynomial}\index{polynomial!Riccati}
  ubiquitous in control theory.
 Second, we have separated the variables into {\it two classes}
  of variables,\index{variables!classes}\index{variables!mixed}
  the $a$ variables and the $x$ variable(s); 
thus $p \in\RR\mathord{<} a, x=x^{\ss} \mathord{>}$.  In applications, the
  $a$ variables can be chosen to represent known (system parameters),
 while the $x$ variables are unknown(s).    Of course,
  it could be that some of the $a$ variables are symmetric
 and some free and ditto for the $x$ variables. 
\eexa

\bexa
Various directional
derivatives of $p$ in  \eqref{2b} are    
$$
D_{x_1}p(x) [h_1] =
 h_1^\ss x_1  +  x_1^\ss h_1    
 +  \frac{3}{4}  h_1  x_2 x_1^\ss  +  \frac{3}{4}  x_1  x_2 h_1^\ss,
 \qquad
D_{x_2}p(x) [h_2] =
 h_2     
 +  \frac{3}{4}  x_1  h_2  x_1^\ss,
$$
$$
D_{x}p(x) [h] =
 h_1^\ss x_1  +  x_1^\ss h_1 + h_2    
 +  \frac{3}{4}  h_1  x_2  x_1^\ss  +  \frac{3}{4}  x_1  x_2 h_1^\ss
 +  \frac{3}{4}  x_1  h_2 x_1^\ss,
$$
\eexa

  Continuing with the variable class  warfare, consider
  the following matrix-valued example.
\bexa
Let
\bes
L(a_1,a_2,x)= 
\begin{bmatrix}
 a_1x +xa_1^\ss & a_2^\ss x     \\
 xa_2 &   1 \\
\end{bmatrix}.
\ees
  We consider 
  $L\in \RR^{2\times 2}\mathord{<} a, x=x^{\ss}\mathord{>}$; i.e.,
 the $a$ variables are free, and the $x$-variables symmetric. 
 Note that $L$ is  linear in $x$ if we consider $a_1,a_2$ fixed.
 Of course, if $a_1,a_2$ and $x$ are all scalars, then
 using 
Schur complements
tells us  there is
  a close relation between $L$ in this example and
 the Riccati of the previous example. 
\eexa

\subsection{\Nc \ rational functions}
 While it is possible to define nc functions 
  \cite{Tay73,AIM,Voi04,Voi10,Pop06,Pop10,KVV-,HKM10a,HKM10b},
 in this section we content ourselves
 with a relatively informal
  discussion  of  nc rational functions 
 \cite{Coh95,Coh06,HMV06,KVV09}.

\subsubsection{Rational functions, a gentle introduction}
  Noncommutative \df{rational expressions}
  are obtained by allowing
  inverses of polynomials.  An example 
  is the discrete time algebraic Riccati equation (DARE)
\bes
  r(a,x) = a_1^{\ss} x a_1 - (a_1^{\ss}xa_2) a_1 (a_3 + a_2^{\ss} x a_2)^{-1} (a_2^{\ss}xa_1) + a_4,
 \ \ \ \ x=x^{\ss}.
\ees
 It is a rational expression in the free variables $a$ 
 and the symmetric variable $x$,
 as is $r^{-1}$.  
 An example, in free variables, which arises in operator theory is
\beq
\label{eq:s}
  s(x)= x^\ss(1-xx^\ss)^{-1}.
\eeq

 Thus, we define (scalar) \df{nc rational expressions} 
 for free nc variables $x$  by
starting with nc polynomials and then applying
successive arithmetic operations - addition, multiplication, and
inversion. We emphasize that an expression includes the order in
which it is composed and no two distinct expressions are
identified, e.g., $(x_1)+(-x_1)$, $(-1)+(((x_1)^{-1})(x_1))$, and
$0$ are different nc rational expressions.

 Evaluation on polynomials naturally extends to rational expressions. 
 If $r$ is a rational expression in free variables and $X\in (\Mnns)^g$, 
 then $r(X)$ is defined - in the obvious way - as long as any inverses appearing actually exist.
 Indeed, our main interest is in the evaluation of a rational expression. 
 For instance, for the polynomial $s$ above in one free variable, $s(X)$ is
  defined as long as $I-XX^\ss$ is invertible and in this case,
\bes
  s(X)= X^\ss (I-XX^\ss)^{-1}. 
\ees
 Generally, a 
nc rational expression $r$ can be evaluated on a
$g$-tuple $X$ of $n \times n$ matrices in its \df{domain of
regularity}, $\dom{r}$, which is defined as the set of all
$g$-tuples of square matrices of all sizes such that all the
inverses involved in the calculation of $r(X)$ exist. For example, if
$r=(x_1 x_2 - x_2 x_1)^{-1}$ then $\dom{r} = \{ X=(X_1,X_2) \colon \det
(X_1X_2-X_2X_1)\neq 0\}$. 
 We assume that $\dom{r} \neq \emptyset.$ 
 In other words, when forming nc rational expressions
 we never invert an expression that is nowhere invertible.
 
  Two rational
 expressions $r_1$ and $r_2$ are {\it equivalent}\index{rational expressions!equivalent}
 if $r_1(X)=r_2(X)$ at any $X$ where both are defined.  For instance, for the
 rational expression $t$ in one free variable,
\[
  t(x) = (1-x^\ss x)^{-1} x^\ss,
\]
 and $s$ from equation \eqref{eq:s}, 
 it is an exercise to check that $s(X)$ is defined if and only if $t(X)$ is and moreover in
 this case $s(X)=t(X)$. Thus $s$ and $t$ are equivalent rational expressions. 
 We call an equivalence class of rational expressions a \df{rational function}.\index{nc rational function}\index{rational function!nc}
 The set of all rational functions will be denoted by $\freef$.

Here is an interesting example of an nc rational 
function with nested inverses. It is taken from \cite[Theorem 6.3]{Ber76}.

\bexa
Consider two free variables $x,y$. 
For any $r\in\freey$ let 
\beq\label{eq:preBerg}
W(r):= c\big(x, c ( x,r)^2\big) \cdot
c\big(x, c ( x,r)^{-1}\big) ^{-1}\in\freey.
\eeq
Recall that $c$ denotes the commutator \eqref{eq:commutator}.
Bergman's nc rational function is given by:\index{rational function!Bergman}
\beq\label{eq:bergRat}
b:= 
W(y)\cdot W\big( c(x,y) \big)\cdot
W\Big( c\big(x,c(x,y)\big)^{-1} \Big)\cdot
W\Big( c\big(x,c(x,c(x,y))\big)^{-1} \Big)
\in\freey.
\eeq
\eexa

\bexe
Consider the function $W$ from \eqref{eq:preBerg}.
Let $R,X$ be $n\times n$ matrices and assume 
$c\big(X, c ( X,R)^{-1}\big)$ exists and is invertible. Prove:
\ben[\rm(1)]
\item
If $n=2,$ then $W(R)=0$.
\item[\rm (2)]
If $n=3,$ then $W(R)=\det (c(X,R)).$
\een
\eexe

\bexe\label{exe:berg}
Consider Bergman's rational function \eqref{eq:bergRat}.
\ben[\rm(1)]
\item
Show that on a dense set of $2\times 2$ matrices $(X,Y)$,
$b(X,Y)=0$.
\item
Prove that on a dense set of $3\times 3$ matrices $(X,Y)$,
$b(X,Y)=1$.
\een
\eexe

The moral of 
Exercise \ref{exe:berg} is that, unlike in the case of
polynomial identities, a nc rational function that
vanishes on (a dense set of) $3\times 3$ matrices need not vanish
on (a dense set of) $2\times 2$ matrices.

\subsubsection{Matrices of Rational Functions; $LDL^\ss$}
One of the main ways nc rational functions\index{rational function!matrix} occur
in systems engineering is in the manipulation of matrices of polynomials.
Extremely important is the  {\it $LDL^{\Ms}$ decomposition}.\index{LDL@$LDL^\ss$ decomposition} 
Consider the  $2 \times 2$ matrix with nc entries
$$
M = \begin{bmatrix}
a & b^{\Ms} \\
b & c
\end{bmatrix}
$$
where $a=a^\ss$. 
The entries themselves could be nc polynomials,
or even rational functions.  If $a$ is not zero,
then $M$ has 
the following decomposition
\bes
M = LDL^\Ms=
\bmat
I & 0 \\
b a^{-1} & I
\emat \bmat
a & 0 \\
0 & c - ba^{-1}b^{\Ms} 
\emat
\bmat
I & a^{-1}b^{\Ms} \\
0 & I
\emat. 
\ees
 Note that this formula holds in the case that 
  $c$ is itself a (square) matrix nc rational function
  and $b$ (and thus $b^\ss$) are vector-valued nc rational
  functions. 
On the other hand, if both $a=c=0$, then $M$ is the block matrix,
\[
  M=\begin{bmatrix} 0 & b \\ b^\ss & 0 \end{bmatrix}.
\]

If $M$ is a  $k \times k$ matrix then iterating this procedure 
produces a decomposition of a permutation  $\Pi M\Pi^\ss$
of $M$ of the form
 $\Pi M\Pi^\ss =LDL^\ss$ where 
  $D$ and $L$ have the form
\begin{equation}
 \label{eq:DinLDLT}
  D =\begin{bmatrix} d_1 & 0 & 0 & 0 & 0 & 0 &0 \\
   \vdots & \ddots & 0 & 0 &\cdots & 0 &0 \\ 
     0 & \cdots & d_k & 0 & \cdots & 0 &0  \\
   0 & \dots & 0 & D_{k+1} & \cdots & 0 &0  \\ 
   \vdots & \cdots & \vdots & \vdots & \ddots &0 &0\\
   0 & \cdots &0 & 0  &\cdots  & D_\ell &0\\
   0 & \cdots & 0 & 0 & \cdots &0& E\end{bmatrix}
\end{equation}
 and $L$ has the form,
\begin{equation}
 \label{eq:LinLDLT}
  L= \begin{bmatrix} 1 & 0 & 0 & 0 & 0 & 0 & 0  \\
   * & \ddots & 0 & 0 &0 & 0 \\ * & * & 1 & 0 & 0 & 0 & 0  \\
   * & * & * & I_2 & 0 & 0 & 0  \\ * & * & * & * & \ddots &0 & 0 \\
   * & * & * & *  & * & I_2 & 0 \\
   * & * & * & *  & * & *  &  I_a\end{bmatrix},
\end{equation}
 where
  $d_j$ are symmetric rational functions, and
  the $D_j$ are nonzero $2\times 2$ matrices of the form
\[
  D_j =\begin{bmatrix} 0 & b_j \\ b_j^\ss & 0 \end{bmatrix},
\]
  $E$ is a square $0$ matrix (possibly of size $0\times 0$ - so absent),
 and
 $I_2$ is the $2\times 2$ identity and the $*$'s represent
 possibly nonzero rational expressions (in some cases matrices
 of rational functions), some of the $0$s are zero matrices (of the appropriate
 sizes), and $a$ is the dimension of the space that $E$ acts upon.   
 The permutation $\Pi$ is necessary in cases where the procedure
 hits a 0 on the diagonal, necessitating a permutation to bring a nonzero
 diagonal entry into the ``pivot'' position.

\begin{theorem}
\label{thm:ldlpos}
  Suppose $M(x)\in \RR \plangle\x\prangle^{\ell \times \ell}$ is symmetric,
  and $\Pi M\Pi^\ss =LDL^\ss$ where $L,D$ are $\ell\times\ell$
  matrices with nc rational entries
  as in equations \eqref{eq:LinLDLT} and
  \eqref{eq:DinLDLT} and $L$ respectively.
  If $n$ is a positive integer and $X\in\gtupn$
  is in the domains of both $L$ and $D$, then 
  $M(X)$ is positive semidefinite   
  if and only if $D(X)$ is positive semidefinite.
\end{theorem}

\begin{proof}
  The proof is an easy exercise based on the fact that 
  a square block lower triangular matrix whose diagonal blocks
  are invertible is itself invertible.  In this case,
  $L(X)$ is block lower triangular with the $n\times n$
  identity $I_n$  as each diagonal entry. Thus $M(X)$
  and $D(X)$ are congruent, so have the same number of negative
  eigenvalues.
\end{proof}

\begin{remark}\rm
  Note that if $D$ has any $2 \times 2$ blocks $D_j$,
  then $D(X)\succeq 0$ if and only if each $D_j(X)=0$.
  Thus, if $D$ has any $2\times 2$ blocks, generically
  $D(X)$, and hence $M(X)$, is not positive semidefinite
  (recall we assume,  without loss of generality that $D_j$
  are not zero). 
\end{remark}

\subsubsection{More on rational functions}
 \label{sec:ratgitmo}
The matrix positivity and convexity properties of nc rational functions go
just like  those for polynomials.\index{rational function!convex}\index{rational function!positive}
One only tests a rational function $r$ on matrices $X$ in its domain of regularity.
The definition of directional derivatives goes as before
and it is easy to compute them formally. 
There are issues of equivalences 
which we avoid here, instead referring the reader to
\cite{Coh95,KVV09} or our treatment in \cite{HMV06}.

We emphasize that proving the assertions above
takes considerable effort, because of dealing with the equivalence
relation. In practice one works with rational
 expressions, and calculations with nc rational
 expressions themselves are straightforward.
  For instance, computing the  derivative of a symmetric nc
rational function $r$ leads to an expression of the form 
$$
Dr(x)[h] = \text{symmetrize} \left[ \sum\limits_{\ell=1}^k
a_\ell(x) h b_\ell(x) \right],
$$
where $a_\ell, b_\ell$ are nc rational functions of $x$,
and the symmetrization of a (not necessarily symmetric) rational
 expression $s$ is $\frac{s+s^\ss}{2}$.

\subsection{Exercises}
 \label{sec:2exes}
Section \ref{sec:compAlg} gives a very brief
introduction on nc computer algebra and some might enjoy playing with 
computer algebra in working some of these exercises.

  Define for use in later exercises the nc polynomials
\bes
\begin{split}
p&= x_1^2x_2^2-x_1x_2x_1x_2- x_2x_1x_2x_1-x_2^2x_1^2	\\
q&= x_1x_2x_3+x_2x_3x_1+x_3x_1x_2- x_1x_3x_2-x_2x_1x_3-x_3x_2x_1	\\
s& =x_1x_3x_2-x_2x_3x_1.
\end{split}
\ees

\def\inv{^{-1}}

\bexe \mbox{}\par
\ben[\rm (a)]
      \item
What is the derivative with respect to $x_1$
in direction $h_1$ of $q$ and $ s$ ?\index{directional derivative}
      \item
 Concerning     the formal derivative with respect to $x_1$
in direction $h_1$. 
\ben[\rm (i)]
\item
Show the derivative of 
$r(x_1)= {x_1}^{-1}$  is  $-x_1\inv  h_1 x_1\inv$.
\item
What is the derivative of
$u(x_1,x_2)=   x_2 (1 + 2 x_1)^{-1}$  ?
\een
\een
      \eexe

\bexe
\label{exe2:eval}
Consider the polynomials $p,q,s$ and rational functions $r,u$ from above.
\ben[\rm(a)]
\item
Evaluate the polynomials $p,q,s$ on some matrices
of size $1\times 1$, $2\times 2$ and $3\times 3$.
\item
Redo part {\rm (a)} for the rational functions $r,u$.
\een
Try to use Mathematica or MATLAB.
\eexe

\bexe
 Show $c=x_1x_2-x_2x_1$ is not symmetric,
  by finding $n$ and $X=(X_1,X_2)$ such
  that $c(X)$ is not a symmetric matrix.\index{commutator}
\eexe

\bexe
Consider the following polynomials in two  and three  variables,
respectively:\index{polynomial identity} 
\[
\begin{split}
h_1 & = c^2 =
(x_1 x_2)^2 - x_1  x_2^2  x_1 - x_2  x_1^2  x_2 + 
 (x_2 x_1)^2, \\
h_2 & = h_1 x_3-x_3 h_1.
\end{split}
\]
\ben[\rm(a)]
\item
Compute $h_1(X_1,X_2)$ and $h_2(X_1,X_2,X_3)$ for 
several choices of $2\times 2$ matrices $X_j$.
What do you find? Can you formulate and prove a statement?
\item
What happens if you plug in $3\times 3$ matrices into $h_1$ and $h_2$?
\een
\eexe

\bexe
\label{exe:geomHess}
Prove that a symmetric nc polynomial $p$ is matrix convex
if and only if the Hessian $p^{\prime\prime}(x)[h]$ is matrix positive,
by completing the following exercise.\index{polynomial!convex}\index{polynomial!positive}

Fix $n$, suppose $\ell$ is a positive linear functional
on $\Snn$, and consider 
$$ f= \ell \circ p :\gtupn\to\RR.$$

\ben[\rm(a)]
\item
Show $f$ is convex if and only if 
 $\frac{d^2 f(X + t H)}{dt^2} \geq 0$
at $t=0$ for all $X, H \in \gtupn$.
\een
\noindent
  Given $v\in\RR^n$, 
  consider the linear functional $\ell(M):= v^{\ss} Mv$
  and let $f_v = \ell\circ p$. 
\ben[\rm(b)]
\item {\em Geometric}:  
  Fix $n$. Show, each $f_v$ satisfies the  convexity inequality
  if and only if $p$  satisfies the   convexity inequality
   on $\gtupn$; and
\item {\em Analytic}: show,  for each $v\in\RR^n$,
  $f_v^{\prime\prime}(X)[H] \geq  0$ for every $X,H\in\gtupn$ 
if
  and only if $p^{\prime\prime}(X)[H]\succeq0$ for 
  every $X,H\in\gtupn$. 
\een
\eexe

\bexe
For $n\in\NN$ let
\bes
s_n = \sum_{\tau\in {\rm Sym}_n} \sign(\tau) x_{\tau(1)} \cdots x_{\tau(n)}
\ees
be a polynomial of degree $n$ in $n$ variables. Here {\rm Sym}$_n$ denotes
the symmetric group on $n$ elements.\index{polynomial identity}
\ben[\rm (a)]
\item
Prove that $s_4$ is a polynomial identity for $2\times 2$ matrices.
That is, for any choice of $2\times 2$ matrices $X_1,\ldots,X_4$,
we have
$$
s_4(X_1,\ldots,X_4)=0.
$$
\item
Fix $d\in\NN$. Prove that there exists a nonzero polynomial $p$
vanishing on all tuples of $d\times d$ matrices.
\een
\eexe

 Several of the next
 exercises use a version of the shift operators on Fock space. 
 With $g$ fixed, the corresponding \df{Fock space},
  $\cF=\cF_g$,  is the Hilbert space
 obtained from $\RR\ax$ by declaring the words to be an orthonormal 
 basis; i.e., if $v,w$ are words, then
\[
  \langle v,w\rangle = \delta_{v,w},
\]
  where $\delta_{v,w}=1$ if $v=w$ and is $0$ otherwise.
 Thus $\cF_g$ is the closure of $\RR\ax$ in this inner product. 
 For each $j$, the operator $S_j$ on $\cF_g$ densely defined by
 $S_j p= x_j p$, for $p\in\RR\ax$ is an isometry (preserves the
 inner product) and hence extends to an isometry on all of
 $\cF_g$.  Of course, $S_j$ acts on an infinite dimensional
 Hilbert space and thus is not a matrix.

\bexe
\label{exe:fock} 
  Given a natural number $k$, 
  note that $\RR\ax_k$ is a finite dimensional
  (and hence closed) subspace of $\cF=\cF_g$.
  The dimension of $\RR\ax_k$ is 
\beq\label{eq:sigma}
  \sigma(k) = \sum_{j=0}^k g^j. 
\eeq
  Let $V:\RR\ax_k\to \cF$ denote the inclusion and
\[
   T_j = V^\ss S_j V.  
\]
  Thus $T_j$ does act on a finite dimensional space,
  and $T=(T_1,\dots,T_g)\in (\RR^{n\times n})^g$, for $n=\sigma(k)$. 
\ben[\rm (a)]
 \item  Show, if $v$ is a word of length at most $k-1,$ then
\[
   T_j v = x_j v;
\] 
 and $T_jv=0$ if the length of $v$ is $k$. 
\item
   Determine $T_j^\ss;$ 
\item
    Show, if $p$ is a nonzero polynomial of degree at most $k$ and
  $Y_j = T_j+T_j^\ss$, then $p(Y)\emptyset \ne 0;$
\item 
  Conclude, 
  if, for every $n$ and $X\in\gtupn$,  $p(X)=0,$ then $p$ is $0$.
\een
\eexe

Exercise \ref{exe:fock} shows there are no nc polynomials
vanishing on all tuples of (symmetric) matrices of all sizes.
The next exercise 
will lead the reader through an alternative proof
inspired by standard methods of polynomial identities.\index{polynomial identity}

\bexe\label{exe:PI}
 Let $p\in\RR\ax_n$ be an analytic 
 polynomial that vanishes on 
 $(\RR^{n\times n})^g$ $($same fixed $n)$. 
Write $p=p_0+p_1+\cdots+p_n$, where $p_j$ is
the homogeneous part of $p$ of degree $j$.
\ben[\rm(a)]
\item\label{it:1}
Show that $p_j$ also vanishes on $(\RR^{n\times n})^g$.
\item
A polynomial $q$ is called multilinear if it is homogeneous of degree
one with respect to all of its variables. Equivalently, 
each of its monomials
contains all variables exactly once, i.e.,
\bes
q=\sum_{\pi\in S_n} \alpha_\pi X_{\pi(1)} \cdots X_{\pi(n)}.
\ees
Using the staircase matrices $E_{11},E_{12},E_{22},E_{23},\ldots,
E_{n-1\, n},E_{nn}$ show that a nonzero multilinear polynomial $q$
of degree $n$ cannot vanish on all $n\times n$ matrices.
\item
By \eqref{it:1} we may assume $p$ is homogeneous. By induction on the
biggest degree a variable in $p$ can have, prove that $p=0$.
{\em Hint:} What are the degrees of the variables appearing in
$$
p(x_1+\hat x_1,x_2,\ldots,x_g)- p(x_1,x_2,\ldots,x_g)-
p(\hat x_1,x_2,\ldots,x_g)?
$$
\een
\eexe

\bexe
Redo Exercise {\rm\ref{exe:PI}} for a polynomial
\ben[\rm(a)]
\item
  $p\in\RR\axxs,$ not necessarily analytic,   vanishing on all
tuples of matrices;
\item
  $p\in\RR\ax$  vanishing on all
tuples of symmetric matrices.
\een
\eexe

\bexe
 \label{exe:openfock}
 Show, if $p\in\RR\ax$ vanishes on a nonempty basic open semialgebraic
 set, then  $p=0$.
\eexe

\bexe
 \label{exe:vanish}
  Suppose $p\in\RR\ax$, $n$ is a positive integer and $O\subseteq \gtupn$
  is an open set.  Show, if $p(X)=0$ for each $X\in O$, then
  $P(X)=0$ for each $X\in\gtupn$. 
{\em Hint:} given $X_0\in O$
  and $X\in\gtupn$, consider the matrix valued polynomial,
\[
  q(t)=p(X_0+tX).
\]
\eexe

\bexe\label{exe:ratVanish}
Suppose $r\in \freef$ is a rational function and there is a nonempty
nc basic open semialgebraic set
$\cO\subseteq \dom(r)$ with $r|_{\cO}=0$. Show that $r=0$.
\eexe

\bexe
 \label{exe:openinadim}
  Prove item {\rm (3)} of Proposition {\rm\ref{prop:faith}}.
  You may wish to use Exercises {\rm \ref{exe:vanish}}
  and {\rm\ref{exe:fock}}. 
\eexe

\bexe
Prove the following proposition:

\begin{proposition}
 \label{prop:findimrep}
  If $\pi:\RR\ax \to \RR^{n\times n}$
  is an involution preserving homomorphism, then
  there is an $X\in (\Snn)^g$ such that  $\pi(p)=p(X)$;
  i.e., all finite dimensional representations of
  $\RR\ax$ are evaluations.
\end{proposition}\index{polynomial!evaluation}

\eexe

\bexe
Do the algebra to
show $$x^\ss(1-xx^\ss)^{-1}=(1-x^\ss x)^{-1} x^\ss.$$
  $($This is a key fact used in the model theory for 
  contractions {\rm\cite{NFBK10}}.$)$
\eexe\index{rational function}

\bexe
  Give an example of symmetric $2\times 2$ matrices $X,Y$ such that
  $X\succeq Y\succeq 0$, but $X^2\not\succeq Y^2$.  

  This failure
  of a basic order property of $\RR$ for 
  $\SRnn$ is closely related to the rigid nature of
  positivity and convexity in the nc setting. 
\eexe

\bexe
Antiderivatives.\index{directional derivative}
\ben[\rm(a)]
\item
Is $q(x)[h]= x h + h x$
the derivative of any nc polynomial $p$?
If so what is $p$?
\item
Is $q(x)[h]= h h x +  hxh  + x h h$
the second derivative of any nc polynomial $p$?
If so what is $p$?
\item
Describe in general
which polynomials $q(x)[h]$ are
the derivative of some nc polynomial $p(x)$.
\item
Check you answer against the theory in 
{\rm\cite{GHV+}.}
\een
\eexe

\bexe $($Requires background in algebra$)$ 
Show that $\freef$ is a division ring;
i.e., the nc rational functions form 
a ring in which every nonzero element is invertible.
\eexe\index{rational function}

\bexe
In this exercise we will establish that
it is possible to embed the free algebra
$\RR\mathord{<}x_1,\ldots,x_g\mathord{>}$ into $\RR\axy$ for any $g\in\NN$.
\ben[\rm(a)]
\item
Show that the subalgebra of $\RR\axy$ generated by $xy^n$, $n\in\NN_0$, is free.
\item
Ditto for the subalgebra generated by
$$x_1=x,\quad
x_2=c(x_1,y),\quad x_3=c(x_2,y),\quad \ldots,\quad x_n=c(x_{n-1},y),\ldots.
$$
 Here, as before, $c$ is the commutator,  $c(a,b)=ab-ba$. 
\een
\eexe\index{commutator}

A comprehensive study of free algebras and nc rational functions
from an algebraic viewpoint is developed in \cite{Coh95,Coh06}.

\bexe
\label{it:TV}
  As a hard exercise, numerically verify that the set
\[
  {\rm ncTV}(2)=\{X\in \twotuptwo: 1-X_1^4 -X_2^4 \succ 0\}
\]
 is not convex. That is, find $X=(X_1,X_2)$ and
 $Y=(Y_1,Y_2)$ where $X_1,X_2,Y_1,Y_2$ are
 $2\times 2$ symmetric matrices such that both
\[
  1-X_1^4-X_2^4 \succ 0 \quad\text{and}\quad
  1-Y_1^4-Y_2^4 \succ 0,
\]
 but
\[
 1- \Big(\frac{X_1+Y_1}{2}\Big)^4 - \Big(\frac{X_2+Y_2}{2}\Big)^4 \not\succ 0.
\]
  You may wish to write a numerical search routine. 
\eexe
\index{TV screen}

\section{Computer algebra support}\label{sec:compAlg}
There are several computer algebra packages available
to ease the first contact with free convexity
and positivity. In this section we briefly describe two
of them:
\ben[\rm(1)]
\item
\NCAlgebra running under Mathematica;
\item 
\NCSOStools running under MATLAB.
\een
The former is more universal in that it implements
manipulation with noncommutative variables, including
nc rationals, and several algorithms pertaining to 
convexity. The latter is focused on nc positivity
and numerics.\index{NCAlgebra}\index{NCSOStools}

\subsection{NCAlgebra} 
 
  NCAlgebra \cite{HOMS++} runs under Mathematica and gives it the capability 
of manipulating noncommuting algebraic expressions. 
An important part of the package (which we shall not go into here)
is NCGB, which computes 
noncommutative Groebner Bases and has extensive sorting and display features as well as algorithms for automatically discarding ``redundant'' polynomials.

We recommend the user to have a look at the Mathematica 
notebook\\
\href{http://math.ucsd.edu/~ncalg/NCBasicCommandsDemo.nb}{\tt NCBasicCommandsDemo} available from the NCAlgebra website\\
\centerline{\url{http://math.ucsd.edu/~ncalg/}}
for the basic commands and their usage in NCAlgebra.
Here is a sample.

The basic ingredients are (symbolic) variables,
which can be either noncommutative or commutative.
At present, single-letter lower case variables are noncommutative by default and all others are commutative by default. 
To change this one can employ

\ncalg 
\ncad{SetNonCommutative[listOfVariables]} to make all the
variables appearing in listOfVariables noncommutative. 
The converse is given by

\ncalg \ncad{SetCommutative}.

\bexa
Here is a sample session in Mathematica running NCAlgebra.
\begin{verbatim}
In[1]:= a ** b - b ** a 
Out[1]= a ** b - b ** a

In[2]:= A ** B - B ** A 
Out[2]= 0

In[3]:= A ** b - b ** a
Out[3]= A b - b ** a

In[4]:= CommuteEverything[a ** b - b ** a] 
Out[4]= 0

In[5]:= SetNonCommutative[A, B]
Out[5]= {False, False}

In[6]:= A ** B - B ** A
Out[6]= A ** B - B ** A     

In[7]:= SetNonCommutative[A];SetCommutative[B]
Out[7]= {True} 

In[8]:= A ** B - B ** A
Out[8]= 0     
\end{verbatim}
\eexa

Slightly more advanced  is the NCAlgebra 
command to generate the directional
derivative\index{directional derivative} of a polynomial $p(x,y)$ with respect to $x$, which is denoted by
$D_{x}p(x, y)[h]$:

\ncalg \ncad{DirectionalD[Function $p$, $x$, $h$]},
and is abbreviated 

\ncalg \ncad{DirD}.

\bexa
Consider 
\begin{verbatim} a = x ** x ** y - y ** x ** y\end{verbatim} 
Then
\bev
DirD[a, x, h] = (h ** x + x ** h) ** y - y ** h ** y
\end{verbatim}
or in expanded form,
\bev
NCExpand[DirD[a, x, h]] = h ** x ** y + x ** h ** y - y ** h ** y
\end{verbatim}
Note that we have used 

\ncalg \ncad{NCExpand[Function $p$]} to expand
a noncommutative expression. The command comes with a convenient
abbreviation 

\ncalg \ncad{NCE}.
\eexa

NCAlgebra is capable of much more. For instance,
is a given noncommutative function ``convex''? 
You type in a function of noncommutative variables; the command

\ncalg \ncad{NCConvexityRegion[Function, ListOfVariables]}
 tells you where the (symbolic) Function is convex in the Variables. 
 The algorithm\index{polynomial!convex}\index{rational function!convex}
comes from the paper of Camino, Helton, Skelton, Ye \cite{CHSY03}.

\ncalg $\{L,D,U,P \}$:=\ncad{NCLDUDecomposition[Matrix]}.
Computes the LDU Decomposition of Matrix and returns the
result as a 4 tuple. The last entry is a Permutation matrix
which reveals which pivots were used. If Matrix is symmetric
then $U=L^\ss$.

The NCAlgebra website comes with extensive documentation.
A more advanced notebook with a hands on demonstration of applied
capabilities of the package is \href{http://math.ucsd.edu/~ncalg/DEMOS/DemoBRL.nb}{\tt DemoBRL.nb}; 
it derives the Bounded Real Lemma for a linear system.

\bexe For the polynomials and rational functions 
defined  at the beginning of Section {\rm\ref{sec:2exes}},
use  NCAlgebra to calculate
 \ben[\rm(a)]
\item
\verb=p**q=
   and  
\verb=NCExpand[p**q]=
\item
\bev
NCCollect[p**q, x1]   
\end{verbatim}
 \item \verb=D[p,x1,h1]= and \verb=D[u,x1,h1]=
  \een
 \eexe

\subsubsection{Warning} 
The Mathematica substitute commands 
\verb=/.=, \verb=/>= and \verb=/:>=
are not reliable in NCAlgebra, so a user should use NCAlgebra's 
Substitute command.

\bexa
Here is an example of unsatisfactory behavior of the built-in
Mathematica function.
\begin{verbatim}
In[1]:= (x ** a ** b) /. {a ** b -> c}
Out[1]= x ** a ** b
\end{verbatim}
On the other hand, NCAlgebra performs as desired:
\begin{verbatim}
In[2]:= Substitute[x ** a ** b, a ** b -> c]
Out[2]= x ** c
\end{verbatim}
\eexa

\subsection{NCSOStools}

A reader mainly interested in positivity of 
noncommutative polynomials\index{NCSOStools}\index{polynomial!positive} might be
better served by NCSOStools \cite{CKP++}.
NCSOStools is an open source MATLAB toolbox for
\ben[\rm (a)]
\item basic symbolic computation with polynomials in noncommuting variables;
\item constructing and solving sum of hermitian squares (with commutators) programs for polynomials in noncommuting variables.
\een
It is normally used in combination with standard semidefinite programming software to solve these constructed LMIs.

The NCSOStools website\\
\centerline{\url{http://ncsostools.fis.unm.si}}
contains  documentation and a demo notebook
\href{http://ncsostools.fis.unm.si/documentation/NCSOStoolsdemo}{\tt NCSOStoolsdemo} to give the user a gentle introduction into its features.

\bexa
Despite some ability to manipulate symbolic expressions, MATLAB cannot
handle noncommuting variables. They are implemented in NCSOStools.

\ncsos \ncsosd{NCvars} $x$ introduces a noncommuting variable $x$ into the
workspace.
\eexa

NCSOStools is well equipped to work with commutators and
sums of (hermitian) squares.
Recall: a \df{commutator} is an expression of the form
$fg-gf$.\index{commutator}

\bexe
Use NCSOStools to check whether the polynomial
$x^2yx+yx^3-2xyx^2$
is a sum of commutators. $(${\em Hint}: Try the \ncsosd{NCisCycEq} command.$)$ If so, can you find such an expression?
\eexe

Let us demonstrate an example with sums of squares.\index{sum of squares} 

\bexa
Consider
\begin{verbatim}
f = 5 + x^2 - 2*x^3 + x^4 + 2*x*y + x*y*x*y - x*y^2 + x*y^2*x
    -2*y + 2*y*x + y*x^2*y - 2*y*x*y + y*x*y*x - 3*y^2 - y^2*x + y^4
\end{verbatim}
Is $f$ matrix positive? By Theorem \ref{thm:ncsos} it suffices to check whether
$f$ is a sum of squares. This is easily done using

\ncsos \ncsosd{NCsos$(f)$}, which checks the polynomial $f$ is a sum of squares.
Running NCsos$(f)$ tells us that $f$ is indeed a sum of squares.
What NCSOStools does, is transform this question into a semidefinite
program (SDP) and then calls a solver.
NCsos comes with several options. Its full command line is
\begin{center}
\verb+[IsSohs,X,base,sohs,g,SDP_data,L] = NCsos(f,params)+
\end{center}
The meaning of the output is as follows:
\ben [$\bullet$]
\item
{\tt IsSohs} equals 1 if the polynomial $f$ is a sum of hermitian squares and $0$ otherwise;
\item
{\tt X} is the Gram matrix solution of the corresponding SDP returned by the solver;
\item
{\tt base} is a list of words which appear in the SOHS decomposition;
\item
{\tt sohs} is the SOHS decomposition of $f$;
\item
{\tt g} is the NCpoly representing $\sum_i m_i^\ss m_i$;
\item
\verb=SDP_data= is a structure holding all the data used in SDP solver;
\item
{\tt L} is the operator representing the dual optimization problem (i.e., the dual feasible SDP matrix).
\een
\eexa

\bexe
Use NCSOStools to 
compute the smallest eigenvalue $f(X,Y)$ can attain for a pair
of symmetric matrices $(X,Y)$. Can you also find a minimizer pair
$(X,Y)$?
\eexe

\bexe\label{exe:sturm}
Let $f=y^2+(xy-1)^\ss(xy-1)$. Show that
\ben[\rm(a)]
\item
$f(X,Y)$ is always positive semidefinite.
\item
For each $\eps>0$ there is a pair of symmetric matrices $(X,Y)$
so that the smallest eigenvalue of $f(X,Y)$ is $\eps$.
\item
Can $f(X,Y)$ be singular?
\een
\eexe

The moral of Example \ref{exe:sturm} is that even if an nc polynomial
is bounded from below, it need not attain its minimum.

\bexe
Redo the Exercise {\rm\ref{exe:sturm}} for
$
  f(x) = x^\ss x +(xx^\ss-1)^\ss (xx^\ss-1).
$
\eexe

\section{A Gram-like representation}
 \label{sec:middle-border}

The next two sections are devoted to a
powerful representation of quadratic functions $q$ in nc variables
which takes a strong form when $q$ is matrix positive;
we call it a {\em\qpos}.
Ultimately we shall apply this to
$q(x)[h]=p''(x)[h]$ and show that if $p$ is matrix convex
(i.e., $q$ is matrix positive), then $p$ has degree two.
We begin by illustrating our grand scheme with examples.

\subsection{Illustrating the  ideas}

\mbox{}
\bexa
The (symmetric) polynomial $p(x)= x_1 x_2 x_1 +x_2 x_1 x_2$ (in symmetric
  variables) 
has Hessian $q(x)[h]=p''(x)[h]$ which is homogeneous quadratic in $h$
and is
$$q(x)[h] = 2   h_1  h_2  x_1   +  2  h_1  x_2  h_1
  +  2 h_2  h_1  x_2 + 2 h_2  x_1  h_2  +  2 x_1  h_2  h_1 + 2 x_2 h_1 h_2.
  $$
  We can write  $q$ in the form
  $$q(x)[h] =
  \begin{bmatrix} h_1 & h_2 &x_2 h_1  & x_1h_2
  \end{bmatrix}
  \begin{bmatrix}
  2x_2 & 0 & 0 & 2 \\
  0 & 2x_1 & 2 & 0 \\
   0 & 2 & 0 & 0 \\
  2 & 0 & 0 & 0
  \end{bmatrix}
  \begin{bmatrix} h_1 \\ h_2 \\ h_1 x_2 \\ h_2x_1
  \end{bmatrix}.
  $$
\eexa

The representation of $q$ displayed above is of the form
$$ q(x)[h] =
V(x)[h]^\ss Z(x) V(x)[h]
$$
where $Z$ is called the \df{middle matrix} (MM) and
$V$ the \df{border vector} (BV).
\index{MM}
\index{BV}
The MM does not contain $h$.
The BV is linear in $h$ with $h$ always on the left.
In Section \ref{sec:2} we define this border vector-middle matrix (BV-MM)
 \index{BV-MM}\index{border vector-middle matrix}
  representation
generally for nc polynomials $q(x)[h]$ which are homogeneous of degree two 
 in the $h$ variables.
 Note the entries of the BV are distinct monomials.

\bexa
 \label{ex:xyxy}
Let $
p=
  x_2  x_1    x_2    x_1  + x_1  x_2  x_1 x_2.
$
Then 
\begin{multline*}
  q= p^{\prime\prime}=2 h_1h_2x_1x_2 + 2 h_1x_2  h_1  x_2 + 2 h_1 x_2  x_1 h_2  +
  2 h_2 h_1 x_2 x_1 + 2 h_2x_1h_2 x_1 + 2 h_2  x_1  x_2  h_1 \\ 
  +  2 x_1h_2 h_1x_2 + 2 x_1 h_2  x_1  h_2 + 2 x_1  x_2  h_1  h_2 +
  2 x_2  h_1  h_2  x_1 + 2 x_2  h_1  x_2  h_1 + 2 x_2  x_1  h_2  h_1.
 \end{multline*}
The BV-MM representation for $q$ is
$$
 q= 
 \left[ h_1 \;\; h_2 \;\; x_2h_1 \;\; x_1h_2 \;\; x_1x_2h_1 \; \; x_2x_1h_2 \right]
\begin{bmatrix}
0 & 2x_2x_1 & 2x_2 & 0 & 0 & 2 \\
2 x_1x_2 & 0 & 0 & 2x_1 & 2 & 0 \\
2x_1 & 0 & 0 & 2 & 0 & 0 \\
0 & 2x_2 & 2 & 0 & 0 & 0 \\
0 & 2 & 0 & 0 & 0 & 0 \\
  2 & 0 & 0 & 0 & 0 & 0
  \end{bmatrix}
  \begin{bmatrix} h_1 \\ h_2 \\ h_1x_2 \\ h_2x_1 \\ h_1x_2x_1
\\ h_2x_1x_2 \end{bmatrix}
$$
\eexa

\bexa
\label{exa:MMintro}
In the one variable case with $h_1= h_1^\ss$ we abbreviate $h_1$ to $h$.
Fix some nc variables not necessarily symmetric
$w:= (a,b,d,e) $ and consider
\begin{equation}\label{ex:pre4.2}
q(w)[h]:=
h a h + e^\ss h b h + h b^\ss h e + e^\ss h dh e.
\end{equation}
which is a quadratic function of $h$.
It can be written in the BV-MM form
\begin{equation}\label{ex:4.2}
q(w)[h]=
\begin{bmatrix} 
h & e^\ss h \end{bmatrix}
\begin{bmatrix} 
a & b^\ss \\ b & d   \end{bmatrix}
\begin{bmatrix}
h \\ h e \end{bmatrix}.
\end{equation}
The representation is unique.

Observe  \eqref{ex:4.2} contrasts strongly
with the commutative case wherein \eqref{ex:pre4.2}
takes the form
\begin{equation*}
q(w)[h] = h ( a + e^\ss b + b^\ss e + e^\ss d e) h.
\end{equation*}
\eexa

\begin{example}
\label{ex:x5}
  \rm
  The Hessian of $p(x)=x^4$ is
\beq\label{eq:hess4}
\begin{split}
q(x)[h]:=p^{\prime\prime}(x)[h]=\  & 
2(x^2 h^2 + xh^2 x + h^2x^2)\\
 &+ 
 2(xhxh + hxh x )\\
 &+
 hx^2h,
\end{split}
\eeq
  a polynomial
that is homogeneous of degree two in $x$ and homogeneous of degree two in 
$h$
that can be expressed as
$$
q(x)[h]=2 \begin{bmatrix} h & xh & x^2 h  \end{bmatrix}
\begin{bmatrix}
x^2& x &1\\
x&1&0\\
1&0&0\\
\end{bmatrix}
\begin{bmatrix}
h\\hx\\hx^2\end{bmatrix}.
$$
\end{example}

Notice that the contribution of the main antidiagonal
 of the MM for $q$ in Example \ref{ex:x5} 
(all $1's$) corresponds to  the right hand side 
 of  first line of \eqref{eq:hess4}.
  Indeed,  each antidiagonal corresponds to a line of 
\eqref{eq:hess4}.

\bexe In Example {\rm\ref{ex:x5}},
for which symmetric matrices
  $X$ is $Z(X)$ positive semidefinite?
\eexe

\bexe
What is the MM $Z(x)$ for $p(x)= x^3$?
For which symmetric matrices
  $X$ is $Z(X)$  positive semidefinite?
\eexe

\bexe
Compute  middle matrix representations using \NCAlgebrah.
The command is
\begin{quote}
{$ \{lt, mq, rt\}  =  $\ncad{NCMatrixOfQuadratic[$q$, $\{h, k\}$]} }
\end{quote}
In the output $mq$ is the MM and $rt$ is the BV and
$lt$ is $(rt)^{\ss}$.
For examples, see\\
{\tt NCConvexityRegionDemo.nb}
In the {\tt NC/DEMOS directory}.
\eexe

\subsubsection{The positivity of $q$ vs.~positivity of the MM}
 In this section we let $q(x)[h]$ denote a polynomial
 which is homogeneous of degree two in $h$, but which is not
 necessarily the Hessian of a nc polynomial.
 While we have focused on Hessians, such a $q$ 
 will still have a BV-MM representation. 
{\it So what good is this representation?}
After all one expects  that $q$ could have wonderful
properties, such as positivity, which are
not shared by its middle matrix.
No, the striking thing is that positivity of $q$ implies positivity of
the MM. Roughly we shall prove
what we call the \df{QuadratischePositivstellensatz}, which 
 is essentially Theorem 3.1 of \cite{CHSY03}. 

\begin{theorem}
\label{thm:qposSS}
If the polynomial\footnote{This theorem is true (but not proved here)
for $q$ which are nc rational in $x$.}
  $q(x)[h]$ is homogeneous quadratic in $h,$
then
$q$ is matrix positive if and only if
its middle matrix $Z$ is matrix positive.

  More generally, suppose $\cO$ is a nonempty  
   nc basic open semialgebraic
  set.  If $q(X)[H]$ is positive semidefinite for 
  all $n\in\NN$, $X\in\cO(n)$ and $H\in\gtupn$, then 
  $Z(X)\succeq 0$ for all $X\in\cO$. 
\end{theorem}
We emphasize that, in the theorem, the convention that the terms of the border vector are distinct
is in force. 

To foreshadow
 Section  \ref{sec:qposSS} and to give an idea of
  the proof of Theorem \ref{thm:qposSS}, we illustrate it
on an example in one variable.
This time we use a {\it free} 
rather than symmetric variable since proofs are a bit
easier.

Consider the noncommutative quadratic function $q $
 given by
\begin{equation} \label{eq:NC1.1}
q(w)[h]:= h^\ss b h + e^\ss h^\ss c h + h^\ss c^\ss h e + e^\ss h^\ss a h e
\end{equation}
where $w=(a,b,c,e)$.
   The border vector $V(w)[h]$ and the
coefficient matrix $Z(w)$ with noncommutative entries are
$$ \begin{array}{ccc}
V(w)[h]=\bmat h \\
                      h e \emat
& \qquad {\rm and} \qquad &
Z(w) = \bmat b & c^\ss \\ c & a  \emat,
\end{array} $$
that is, $q$ has the form
\begin{equation*}
q(w)[h] = V(w)[h]^\ss Z(w) V(w)[h] =
\bmat h^\ss & e^\ss h^\ss \emat
\bmat
b & c^\ss \\ c & a  \emat
\bmat
h \\  he \emat.
\end{equation*}

Now, if in equation \eqref{eq:NC1.1} the elements $a$, $b$, $c$, $e$,
$h$ are replaced by matrices in $\RR^{n \times n}$, then the
noncommutative quadratic function $q(w)[h]$
becomes a matrix valued function $q(W)[H]$.
The matrix valued function $q[H]$ is matrix
positive  if and only if $v^\ss q(W)[H] v \geq 0$ for
all vectors $v \in \RR^n$ and all $H \in \RR^{n \times n}$.
Or
equivalently, the following inequality must hold
\begin{equation} \label{eq:4.2.5}
\bmat
v^\ss  H^\ss & v^\ss  E^\ss  H^\ss
\emat
Z
\bmat
H v \\ H  E v \emat \ge 0.
\end{equation}
Let
\bes
y^\ss :=\bmat v^\ss  H^\ss & v^\ss E^\ss  H^\ss
\emat.
\ees
Then \eqref{eq:4.2.5} is equivalent to $y^\ss Z \,y \ge 0$. Now it
suffices to prove that all vectors of the form $y$ sweep $\RR^{2n}$.
This will be completely analyzed in full generality in
Section \ref{appendix:CHSY} but next we give the proof for our simple 
situation.

Suppose for a given $v$, with $n \ge 2$, the vectors $v$ and $Ev$ are
linearly independent. Let \( y = \begin{bmatrix} v_1 \\ v_2
\end{bmatrix}\) be any vector in $\RR^{2n}$, then we can choose $H \in
\RR^{n\times n}$ with the property that $v_1 = H v$ and $v_2 = H
Ev$. It is clear that 
\begin{equation}
 \label{eq:Rhatv}
   {\cR}^v := \left\{\begin{bmatrix} H v\\ HE v
   \end{bmatrix}:  H\in\RR^{n\times n}\right\}
\end{equation}
is all $\RR^{2n}$ as required.

Thus we are
finished unless for all $v$ the vectors $v$ and $Ev$ are linearly
dependent. That is for all $v$, $\lambda_1 (v) v + \lambda_2(v) E v=
0$ for nonzero $\lambda_1 (v)$ and $\lambda_2(v)$.
  Note $\lambda_2(v)
\not= 0$, unless $v=0$.
Set
$\tau(v) := \frac{\lambda_1(v)}{\lambda_2(v)}$,
then the linear dependence becomes $\tau(v) v + E v= 0$
for all $v$. It turns out that
this does not happen unless $E= \tau I$ for some $\tau \in \RR$.
This is a baby case of  Theorem \ref{thm:gendep} which 
 comes later and is a subject unto itself.

To finish the proof pick a $v$ which makes $\cR^v$ equal all of $\RR^{2n}$.
Then $v^{\ss} q(W)[H] v \geq 0$ implies that $Z\succeq0$,
by \eqref{eq:4.2.5}.
\qed

\subsection{Details of the Middle Matrix representation}

   \label{sec:2}
The following
    representation for
   symmetric nc polynomials $q(x)[h]$ that are of degree $\ell$ in
$x$ and
homogeneous of degree two in $h$ is exploited extensively  in this subject:

\begin{equation}\label{eq:rep1}
q(x)[h]=
\bmat
V^\ss_0& V^\ss_1& \cdots & V^\ss_{\ell-1}& V^\ss_\ell
\emat\bmat
Z_{00}&Z_{01}& \cdots&Z_{0, \ell-1}& Z_{0\ell}\\
Z_{10}&Z_{11}& \cdots&Z_{1, \ell-1}&0\\
\vdots&\vdots& \Ddots &\vdots&\vdots\\
Z_{\ell-1, 0}&Z_{\ell-2,1}& \cdots&0&0\\
Z_{\ell 0}&0& \cdots&0&0
\emat
\bmat
V_0\\V_1\\ \vdots\\
V_{\ell-1}\\
 V_\ell\emat,
\end{equation}
where:
\begin{enumerate}
\item[(1)] The degree $d$ of $q(x)[h]$ is  $d=\ell +2$.
\item[(2)]
$V_j=V_j(x)[h]$, $j=0,\ldots,\ell$, is a vector  of
height $g^{j+1}$ whose entries are
monomials of degree $j$ in the $x$ variables
and degree one in the $h$ variables.
The $h$ always appears to the left.
In particular,
$V(x)[h]$ is a vector
of height $g\sigma(\ell)$, where as in \eqref{eq:sigma},
$$
\sigma(\ell)=1+g+\cdots+g^{\ell}\,.
$$

\item[(3)]
$Z_{ij}=Z_{ij}(x)$, is a matrix of size $g^{i+1} \times g^{j+1}$
whose  entries
are polynomials in  the noncommuting variables
$x_1, \dots, x_g$ of degree
$\le \ell-(i+j)$.
In particular, $Z_{i,\ell-i}=Z_{i,\ell-i}(x)$ is a
constant matrix for $i=0,\ldots,\ell$.

\item[(4)]
$Z^\ss_{ij} = Z_{ji}$.
\end{enumerate}

Usually the entries of the vectors $V_j$ are ordered lexicographically.

We note that the vector of monomials, $V(x)[h]$,
might contain monomials that
are not required in the representation of the nc quadratic $q$.
Therefore, we can omit
all monomials from the border vector that are not required.
This gives us a \textit{minimal length} border vector and
prevents extraneous zeros from occurring in the middle matrix.
The matrix $Z$ in the representation \eqref{eq:rep1} will be referred to
as the {\it middle matrix $($MM$)$ of the polynomial $q(x)[h]$}\index{middle matrix}  and the vectors
$V_j=V_j(x)[h]$ with monomials as entries
will be referred to as {\it border vectors $($BV$)$}.\index{border vector} 
It is easy to check that
a minimal length border vector contains distinct monomials
and once the ordering of entries of $V$ is set the MM for
a given $q$ is unique,
see Lemma \ref{lem:MMunique} below.

\begin{example}\rm
  Returning to Example \ref{ex:xyxy}, we have for the MM representation 
  of $q$ that
$$
 V_0 =  \begin{bmatrix} h_1 \\  h_2 \end{bmatrix}, \qquad
  V_1 =  \begin{bmatrix} h_2 x_1 \\ h_1x_2 \end{bmatrix},  \qquad
  V_2 =  \begin{bmatrix} h_1x_2x_1 \\  h_2x_1 x_2 \end{bmatrix} 
$$
 and, for instance,
 $$
  Z_{00} =  \begin{bmatrix} 0 & 2x_2 x_1 \\ 2x_1 x_2 & 0 \end{bmatrix},\qquad
  Z_{01} =  \begin{bmatrix} 2x^2 & 0 \\ 0 & 2x_1 \end{bmatrix}, \qquad
  Z_{02} =  \begin{bmatrix} 0 & 2 \\  2 & 0 \end{bmatrix}.
$$
 Note that generically for a polynomial $q$ in two variables
 the $V_j$ have additional terms. For instance, usually
 $V_1$ is the column 
\[
 \begin{bmatrix} h_1 x_1 \\ h_1x_2 \\ h_2x_1 \\ h_2 x_2 \end{bmatrix}.
\]
  Likewise generically $V_2$ has eight terms.  As for the
 $Z_{ij}$, for instance $Z_{01}$ is generically $2\times 4$.
\end{example}

\begin{lem}
\label{lem:MMunique}
The entries in the middle matrix $Z(x)$ are uniquely
determined by the polynomial $q(x)[h]$ and the border vector $V(x)[h]$.
\end{lem}

\proof Note every monomial in $q(x)[h]$ has the form
$$ m_L h_i m_M h_j m_R.$$
Define
$$
\cR_j:= \{ h_j m: \ m_L h_i m_M h_j m \text{ is a term in }
q(x)[h] \}.
$$
Given the representation $V^\ss ZV$ for $q$, let $E_V$ denote the
monomials in $V$. Then it is clear that each monomial in $E_V$
must occur in some term of $q$, so it appears in  $\cR_j$ for some
$j$. Conversely, each term $h_j m$ in $\cR_j$ corresponds to at
least one term $m_L h_i m_M h_j m$ of $q$, so it must be in $E_V$.

\bexe
 \label{exe:degreebound}
 Consider Equation \eqref{eq:rep1} and
prove the degree bound on the $Z_{ij}$ in {\rm (3)}.
{\em Hint}: Read Example {\rm\ref{exa:Zd4}} first.
\eexe

\bexa
\label{exa:Zd4}
If $p(x)$ is a symmetric polynomial of degree $d=4$
in $g$ noncommuting variables, then the middle matrix $Z(x)$ in the
representation of the Hessian $p^{\prime\prime}(x)[h]$ is
\bes
Z(x) =
\bmat
Z_{00}(x)&Z_{01}(x)& Z_{02}(x)\\
Z_{10}(x)&Z_{11}(x)& 0\\
Z_{20}(x)&0&0
\emat,
\ees
where the block entries $Z_{ij}=Z_{ij}(x)$ have the following structure:
$$
\begin{array}{l}
Z_{00}\quad \textrm{is a}\ \ g \times g \ \
\textrm{matrix with nc polynomial entries of degree}\  \le 2,\\
Z_{01}\quad \textrm{is a}\ \ g \times g^2 \ \
\textrm{matrix with with nc polynomial entries of degree}\  \le 1,\\
Z_{02}\quad \textrm{is a}\ \ g \times g^3 \ \
\textrm{matrix with constant entries}.
\end{array}
$$
All of these are proved merely by keeping track of the degrees.
For example, the contribution of $Z_{02}$ to $p''$
is $V_0^{\ss} Z_{02} V_2$ whose degree is
$$ \deg ( V_0^{\ss} ) +  \deg ( Z_{02} ) +  \deg ( V_2 )= 1  +  \deg ( Z_{02}) +  3  \le 4,$$
so $\deg ( Z_{02} ) = 0$.
\eexa

\subsection{The Middle Matrix of $p^{\prime\prime}$.}
The middle matrix $Z(x)$ of the Hessian $p^{\prime\prime}(x)[h]$ of an nc
symmetric polynomial $p(x)$ plays a key role.
These middle matrices have
a very rigid structure similar to that in Example \ref{ex:x5}.
We illustrate  with an  example and then with  exercises.
\index{Hessian}\index{middle matrix}

\bexa
 \label{exe:warmandfuzzy}
  As a warm up we first illustrate that 
  $Z_{02}(X)=0$ if and only if $Z_{11}(X)=0$ for
  Example \ref{ex:xyxy}. To this end,  observe that
  the contribution of the MM's extreme outer diagonal element
  $Z_{02}$
  to $q$ is as follows
  \bes
    \frac12 V_0(x)[h]^{\ss} Z_{02}(x) V_2(x)[h] =
     \begin{bmatrix} h_1 \\  h_2 \end{bmatrix}^\ss
  \begin{bmatrix} 0 & 2 \\  2 & 0 \end{bmatrix}
  \begin{bmatrix} h_1x_2x_1 \\  h_2x_1 x_2 \end{bmatrix}  
  =  2 h_1 h_2 x_1 x_2   + 2 h_2 h_1 x_2 x_1.
   \ees
 Substitute  $h_j \rightsquigarrow x_j$ and get
 $ 2 x_1 x_2 x_1 x_2   + 2 x_2 x_1 x_2 x_1  
$ which is  $2p(x)$.
That is,
  \[
  p(x) =  \frac12 V_0(x)[x]^{\ss} Z_{02}(x) V_2(x)[x],
\]
  where $V_k(x)[h]$ is the homogeneous, in $x$,  of degree $k$ 
  part of the border vector $V$.
Obviously,  $Z_{02}=0$ implies $p=0$. 
\eexa

\bexe
  Show $p(x)$ can also be obtained from $Z_{11}$ in a similar fashion; i.e.,
\[
  p(x)=\frac12 V_1(x)[x]^\ss Z_{11}(x) V_1(x)[x].
\]
\eexe

\bexe
 \label{exe:heavyLift}
 Suppose $p$ is homogeneous of degree $d$
  and its Hessian $q$ has the border vector middle matrix
  representation $q(x)[h]=V(x)[h]^\ss Z(x)V(x)[h]$. 
\ben[\rm (a)]
\item
  Show,
$$
  p =   \frac{1}{2}V_0(x)[x]^\ss Z_{0\ell} V_{\ell}(x)[x]
$$
with $\ell=d-2$.
Prove this formula for $d=2$, $d=4$.

\item
  Show that likewise,
\[
  p =  \frac{1}{2}V_1(x)[x]^\ss  Z_{1,\ell-1}(x) V_{\ell-1}(x)[x]
\]
  Do not cheat and look this up in {\rm\cite{DGHM}},
  but do compare with Exercise {\rm\ref{exe:degreebound}}.
\een
\eexe

\bexe
 \label{exe:heavy}
  Let $Z$ denote the middle matrix for the Hessian of 
  a nc polynomial $p$. Show, if  
  $i+j=i^\prime + j^\prime$,
 then $Z_{ij}=0$ if and only if $Z_{i^\prime j^\prime} =0$.
\eexe

\subsection{Positivity of the Middle Matrix and the demise of nc convexity}
\label{sec:eigsMM}
 This section focuses on positivity of the middle matrix 
 of a Hessian.

 Why should we focus on the case where $Z(x)$ is positive
 semidefinite? 
In \cite{HMe98} it was shown that a polynomial $p\in\RR\ax$ is matrix convex
if and only if its  Hessian $p^{\prime\prime}(x)[h]$ is positive
(see Exercise \ref{exe:geomHess}).\index{convex polynomial}\index{polynomial!convex}
Moreover, if $Z(x)$ 
 is positive, then  the degree of $p(x)$ is at most two \cite{HM04a}.
The  proof of this degree constraint
given in Proposition \ref{lem:datmosttwo} below
using the more manageable bookkeeping scheme
in this chapter,
begins with the following exercise.

\bexe
 \label{exe:poshess}
Show that
$$\bmat
A&B\\B^\ss&0\emat,
$$
is positive semidefinite if and only if $A\succeq0$ and $B=0$.
 More refined versions of this fact appear as exercises
 later, see    
Exercise {\rm\ref{exe:Fsign}}.
\eexe

As we shall see we need not require our favorite functions be positive
everywhere. It is possible to work locally, namely on an open set.

\begin{proposition}
  \label{lem:datmosttwo}
Let $p=p(x)$ be a symmetric polynomial of degree $d$ in $g$ nc variables
and let $Z(x)$ denote the 
middle matrix $($MM$)$  in the BV-MM representation of the Hessian
$p^{\prime\prime}(x)[h].$ If
   $Z(X) \succeq 0$ for all $ X$ in some nonempty 
  nc basic  open semialgebraic set $\cO,$
   then $d$ is at most two.
\end{proposition}

\begin{proof}
 Arguing by contradiction, 
 suppose $d\ge 3,$ then $p^{\prime\prime}(x)[h]$ is
 of degree $\ell=d-2 \geq 1$ in $x$
  and its middle matrix is of the form
$$Z =
\bmat
Z_{00}&\cdots& Z_{0\ell}\\
\vdots & \Ddots &\vdots\\
Z_{\ell 0}&\cdots&0
\emat.$$
Therefore, $Z(X)$ is of the form
$$
Z(X)=\bmat
A&B\\B^\ss&0\emat,
$$
where $A=A^\ss$ and $B^\ss=\begin{bmatrix}Z_{0\ell}(X) & 0 &\cdots &0
\end{bmatrix}$.
  From Exercise \ref{exe:heavyLift}, $p_d$, the homogeneous 
  degree $d$ part of $p,$ can be reconstructed from
  $Z_{0\ell}$.  Now there is an $X\in\cO$ such
  that $p_d(X)$ is nonzero, as otherwise 
  $p_d$ vanishes on a basic open semialgebraic set
  and is  equal to $0$. 
  It follows that there is an $X\in\cO$ such that 
  $Z_{0\ell}(X)$ is not zero.  Hence $B(X)$ is not zero 
  which implies, by Exercise \ref{exe:poshess}, the contradiction
   that $Z(X)$ is not positive semidefinite.
\end{proof}

We have now reached our goal of showing that convex polynomials have
degree $\leq2$.

\begin{theorem}
  \label{thm:convexd2}
 If $p \in \RR\ax$ is a symmetric polynomial 
 which is convex on a nonempty nc basic open semialgebraic set $\cO$, 
 then it has degree at most two.
\end{theorem}

There is a version of the theorem for free variables; i.e., with
 $p\in\RR\axxs$. 

\begin{proof}
The convexity of $p$ on $\cO$
is equivalent to $p''(X)[H]$
being positive semidefinite for all $X$ in $\cO$, see Exercise 
 \ref{exe:geomHess}.
By the QuadratischePositivstellensatz 
the middle matrix  $Z(x)$ for 
$p^{\prime\prime}(x)[h]$ is positive on $\cO;$ 
that is,
   $Z(X) \succeq 0$ for all $X \in \cO$. 
Proposition
\ref{lem:datmosttwo} implies degree $p$ is at most 2.
\end{proof}

\subsection{The signature of the middle matrix}
   This section introduces the notion of the 
 \df{signature}\index{middle matrix!signature} 
$\mu_\pm( Z(x) )$ of $Z(x)$, the middle matrix
 of a Hessian,  or more generally a 
 polynomial $q(x)[h]$ which is homogeneous of degree two in $h$.

The {\it signature of a symmetric matrix}\index{signature!symmetric matrix} 
$M$ is a triple of integers:
$$ \big(\mu_-( M ), \  \mu_0(M ), \  \mu_+( M ) \big),$$
where $\mu_-(M)$ is the number of negative eigenvalues (counted
  with multiplicity);
 $\mu_+(M)$ is the number of positive eigenvalues; and
 $\mu_0(M)$ is the dimension of the null space of $M$.

\begin{lem}
 \label{lem:middlezeros}
 A  nc symmetric  polynomial $q(x)[h]$ 
  homogeneous of degree two in $h$
has middle matrix  $Z$ of the form in \eqref{eq:rep1}
and $Z$ being   positive semidefinite implies $Z$ is of the form
$$
\bmat
Z_{00}&Z_{01}& \cdots&Z_{0, \lfloor \frac{\ell}{2}\rfloor }& 0&\cdots \\
Z_{10}&Z_{11}& \cdots&Z_{1, \lfloor \frac{\ell}{2}\rfloor}&0&\cdots \\
\vdots&\vdots&\Ddots&\vdots&\vdots&\Ddots\\
Z_{\lfloor \frac{\ell}{2}\rfloor, 0}&Z_{\lfloor \frac{\ell}{2}\rfloor, 1}& \cdots&Z_{\lfloor \frac{\ell}{2}\rfloor, \lfloor \frac{\ell}{2}\rfloor}&0&\Ddots \\
0&0& \cdots&0&0&\Ddots \\
\vdots &\vdots& \Ddots&\Ddots&\Ddots& 
\emat.
$$
\end{lem}

This lemma follows immediately from a much more general
lemma.

\begin{lem}
  \label{lem:U11}
If
$$
E=\bmat
A&B&C\\
B^\ss&D&0\\
C^\ss&0&0
\emat
$$
is a real symmetric matrix, then
$$
\mu_{\pm}(E)\ge \mu_{\pm}(D)+\rank C.
$$
\end{lem}
This can be proved using the $L D L^{\Ms}$
decomposition which we shall not do here
but suggest the reader  apply the $L D L^{\Ms}$
hammer to the following simpler exercise.

\subsection{Exercises}

\bexe
\label{exer:Zanti}
True of False? 
If $p_d$ is homogeneous of degree $d$
and we let $Z$ denote the middle matrix of the Hessian
$p''(x)[h]$,
then for each  $k \leq d-2$
the degree of $Z_{i, k-i}$ is
independent of $i$.
\eexe

\bexe
Redo Exercise {\rm\ref{exe:geomHess}}
for convexity on a
nc basic open  semialgebraic set.
\eexe

\bexe
\label{exe:Fsign}
If $F=\begin{bmatrix}A&C\\C^\ss&0\end{bmatrix}$, then
$\mu_{\pm}(F)\ge \rank C$.
$($If you cannot do the general case, assume $A$ is invertible.$)$
\eexe

\bexe
If $p(x)$ is a symmetric polynomial of degree $d=2$
in $g$ noncommuting variables, then the middle matrix $Z(x)$ in the
representation of the Hessian $p^{\prime\prime}(x)[h]$ is equal to the
$g\times g$ constant matrix  $Z_{00}$.
Substituting $X \in \gtupn$ for $x$ gives
$$
\mu_\pm(Z(X)) \geq  \mu_\pm(Z_{00})
$$
\eexe

\bexe
Let $f\in\RR\ax_{2d}$ and let $V\in\ax_d^{\sigma(d)}$ be a vector
consisting of all words in $x$ of degree $\leq d$. Prove:
\ben[\rm (a)]
\item
there is
a matrix $G\in \RR^{\sigma(d)\times \sigma(d)}$ with
$f=V^\ss G V$
$($any such $G$ is called a \df{Gram matrix} for $f)$;
\item
if $f$ is symmetric, then there is a symmetric Gram matrix for $f$.
\een
\eexe\index{Gram matrix}

\bexe
Find all Gram matrices for 
\ben[\rm(a)]
\item
$f=x_1^4+x_1^2x_2-x_1x_2^2+x_2x_1^2-x_2^2x_1+x_1^2-x_2^2+2x_1-x_2+4$;
\item
$f=c(x_1,x_2)^2$.
\een
\eexe

\bexe
Show: if $f\in\RR\ax$ is homogeneous of degree $2d$, then it has a unique
Gram matrix
$G\in\RR^{\sigma(d)\times \sigma(d)}$.
\eexe

\subsection{A glimpse of history}
  There is a theory of operator monotone and operator convex 
  functions which overlaps with the matrix convex functions
  considered here in the case of one variable.  However, the points of
  view are substantially different, diverging markedly in
  several variables. 
  L\"owner introduced   a class of real analytic
  functions in one real variable called matrix monotone functions,
  which we shall not  define  here.
  L\"owner gave  integral representations
  and these have developed substantially  over the years.
  The contact with convexity came
  when L\"owner's student
  Kraus \cite{K36} introduced matrix convex functions
  $f$ in one variable.
  Such a function $f$ on $[0,\infty) \subseteq \RR$
  can be represented as $f(t) = tg(t)$  with $g$ matrix monotone,
  so the representations for $g$ produce representations  for $f$.
   Hansen has extensive deep work
  on matrix convex and monotone functions
   whose  definition  in several variables is different than the
   one we use here, see \cite{HanT06} or \cite{Han97}.
  All of this gives a beautiful  integral representation
  characterizing matrix convex functions using techniques
  very different from ours.
  An excellent  treatment of  the one variable case is \cite[Chapter 5]{B97}.
  Interestingly, to the best of our knowledge, the one variable
   version of Theorem \ref{thm:convexd2} (\cite{HM04a}) does not seem to be
   explicit in this classical literature.
However, it
   is an immediate consequence of the results
   of \cite{HanT06} where (not necessarily polynomial)
   operator convex functions on an interval are described.
This and  the papers of Hansen and \cite{OSTprept,U02}
   are some of the more recent references
in this line of convexity  history orthogonal to ours.
 
\section{Der QuadratischePositivstellensatz}
\label{sec:qposSS}

In this section we present the proof of the QuadratischePositivstellensatz\index{QuadratischePositivstellensatz}
(Theorem \ref{thm:qposSS}) which is based 
on the fact
that local linear dependence of nc rationals (or nc polynomials) 
implies global linear dependence, a fact itself based on the forthcoming
CHSY Lemma \cite{CHSY03}.\index{CHSY Lemma}

\subsection{The Camino, Helton, Skelton, Ye (CHSY) Lemma}
 \label{appendix:CHSY}
 At the root of the CHSY Lemma \cite{CHSY03} is the following
linear algebra fact:

\begin{lem}
 \label{lem:CHSY-start}
  Fix $n>d.$ If $\{z_1,\dots,z_d\}$ is a linearly independent
  set in $\RR^n$, then the codimension of  \begin{equation*}
   \left\{ \begin{bmatrix} Hz_1 \\ Hz_2 \\ \vdots \\ Hz_d
 \end{bmatrix}:
      H \in \SRnn \right\} \subseteq \RR^{nd}
  \end{equation*}
   is $\frac{d(d-1)}{2}$. 
   It is especially important that
this codimension is  independent
   of $n$.
\end{lem}

 The following exercise  is a 
 variant of the Lemma \ref{lem:CHSY-start} which is easier to prove.
 Thus we suggest attempting it  before launching into the proof of
 the lemma.

\bexe
  Prove if  $\{z_1,\dots,z_d\}$ is a linearly independent
 set in $\RR^n$, then
$$
\left\{ \begin{bmatrix} 
Hz_1 \\ Hz_2 \\ \vdots \\ Hz_d
 \end{bmatrix} :
      H \in \RR^{n\times n} \right\} = \RR^{nd}
$$
{\em Hint}: it goes like the proof of  {\rm\eqref{eq:Rhatv}}.
\eexe

\begin{proof}[Proof of Lemma {\rm\ref{lem:CHSY-start}}]
   Consider the mapping $\Phi:\SRnn\to\RR^{nd}$ given by
 \begin{equation*}
   H \mapsto
      \begin{bmatrix} Hz_1 \\ Hz_2 \\ \vdots \\ Hz_d \end{bmatrix}.
 \end{equation*}
   Since the span of $\{z_1,\dots,z_d\}$ has dimension $d$, it follows
 that
   the kernel of $\Phi$ has dimension $\kappa=\frac{(n-d)(n-d+1)}{2}$
 and
   hence the range has dimension $\frac{n(n+1)}{2}-\kappa$.  To see
 this
   assertion, it suffices to assume that the span of
 $\{z_1,\dots,z_d\}$
   is the span of $\{e_1,\dots,e_d\}\subseteq \RR^n$ (the first
  $d$ standard basis vectors in $\RR^n$).
  In this  case (since $H$ is symmetric) $Hz_j=0$ for all $j$  if and only if
 \begin{equation*}
    H=\begin{bmatrix} 0 & 0 \\ 0 & H^\prime \end{bmatrix},
 \end{equation*}
    where $H^\prime$ is a  symmetric matrix of size $(n-d)\times (n-d)$;
    in other words, this is the kernel of $\Phi$.

   From this we deduce
   that the codimension of the range of $\Phi$ is
 \begin{equation*}
    nd -\Big(\frac{n(n+1)}{2}-\kappa \Big) = \frac{d(d-1)}{2}, 
 \end{equation*}
 concluding the proof.
\end{proof}

Next is a straightforward extension of Lemma \ref{lem:CHSY-start}.

\begin{lem}[\cite{CHSY03}]
 \label{lem:CHSY}
 If $n>d$ and $\{z_1,\dots,z_d\}$ is a linearly independent
 subset of $\RR^n$, then the codimension of
\begin{equation*}
  \Big\{ \oplus_{j=1}^g \begin{bmatrix} H_jz_1 \\ H_jz_2 \\ \vdots \\
 H_jz_d \end{bmatrix} :
     H=(H_1,\dots,H_g)\in \gtupn\Big\} \subseteq \RR^{gnd}
\end{equation*}
 is $g\frac{d(d-1)}{2}$ and is independent of $n$.
\end{lem}

\begin{proof}
See Exercise \ref{exe:lemCHSY}.
\end{proof}

  Finally, the form in which we generally apply the lemma is the
 following.

\begin{lem}
 \label{lem:CHSY-in-action}
  Let $v\in\RR^n$, $X\in \gtupn$.
If the set $\{m(X)v \colon m \in \ax_d \}$ is
 linearly
  independent, then the codimension of
$$
  \{V(X)[H]v \colon H\in \gtupn\}
$$
 is  $g\frac{\kappa(\kappa-1)}{2}$, where
  $\kappa =\sigma(d)=\sum_{j=0}^{d} g^j$ and where 
$$V=\bigoplus_{i=1}^g \bigoplus_{m\in\ax_d} H_i m$$ is the
  border vector associated to $\ax_d$. 
  Again, this codimension is independent of $n$ 
  as it only depends upon the number of variables $g$
  and the degree $d$ of the polynomial. 
\end{lem}

\begin{proof}
  Let $z_m = m(X)v$ for $m\in\ax_d$.
  There are at most $\kappa$ of these. Now apply the previous lemma.
\end{proof}

\subsection{Linear Dependence of Symbolic Functions}
  \label{sec:lindep}
   The main result in this section, Theorem \ref{thm:gendep}
   says roughly that if
   each evaluation of a set $G_1, \ldots G_\ell$
    of rational functions 
    produces linearly dependent matrices, then
   they satisfy a universal linear dependence  
relation. We begin
   with a clean and easily stated consequence of Theorem \ref{thm:gendep}.

In Subsection \ref{subsec:ncsa}
we defined nc basic open  semialgebraic sets.
 Here we define 
a nc  basic  semialgebraic set.  Given matrix-valued 
  symmetric nc polynomials
  $\rho$ and $\tilde{\rho}$, let
\[
        \cD_+^\rho(n) = \{ X \in \gtupn \colon \rho(X)\succ 0\},
\]
 and
\[
        \cD^{\tilde \rho}(n) =\{ X \in \gtupn \colon \tilde{\rho}(X)\succeq 0\}.
  \]
  Then $\cD$ is a \df{nc basic semialgebraic set}\index{semialgebraic set} if there exists
  $\rho_1,\dots,\rho_k$ and $\tilde{\rho_1},\dots,\tilde\rho_{\tilde{k}}$
  such that $\cD=(\cD(n))_{n\in\NN}$ where
\[
  \cD(n) = \big(\bigcap_j \cD_+^{\rho_j}(n)\big) \cap 
\big(\bigcap_j \cD^{\tilde\rho_{\tilde{j}}}(n)\big).
\]

\begin{thm}
     \label{thm:depevals}
      Suppose $G_1$, $\dots$, $G_{\ell}$ 
      are  rational expressions 
and      $\cD$ is a nonempty  nc basic semialgebraic set on 
      which each $G_j$ is defined. 
      If, for each $X \in \cD(n)$
     and vector $v\in\RR^n$ the set 
      $\{G_j(X)v \colon j=1,2,\dots,\ell\}$ is
      linearly dependent,
    then the set $\{G_j(X) \colon j=1,2,\dots,\ell\}$ is linearly dependent on $\cD$, i.e.
    there exists a nonzero $\lambda\in\RR^\ell$ such that
\bes
      0=\sum_{j=1}^\ell \lambda_j G_j(X) \qquad  \text{for  all }X \in   \cD.
\ees
 If, in addition,  $\cD$ 
contains an $\eps$-neighborhood
of $0$ for some $\eps>0$, then
there exists a nonzero $\lambda\in\RR^\ell$ such that
\bes
      0=\sum_{j=1}^\ell \lambda_j G_j.
\ees
\end{thm}

\begin{cor}
     \label{cor:depevalsSimple0}
      Suppose $G_1$, $\dots$, $G_{\ell}$ 
      are  rational expressions. 
      If, for each $n\in\NN$, $X \in \gtupn$,
    and vector $v\in\RR^n$ 
the set 
      $\{G_j(X)v \colon j=1,2,\dots,\ell\}$ is
      linearly dependent,
    then the set $\{G_j \colon j=1,2,\dots,\ell\}$ is linearly dependent,
i.e.,
    there exists a nonzero $\lambda\in\RR^\ell$ such that
\bes
      \sum_{j=1}^\ell \lambda_j G_j=0.
\ees
\end{cor}

\begin{cor}
     \label{cor:depevalsSimple1}
      Suppose $G_1$, $\dots$, $G_{\ell}$ 
      are  rational expressions. 
      If, for each $n\in\NN$ and $X \in \gtupn$,
the set 
      $\{G_j(X) \colon j=1,2,\dots,\ell\}$ is
      linearly dependent,
    then the set $\{G_j \colon j=1,2,\dots,\ell\}$ is linearly dependent.
\end{cor}

The point is that the $\lambda_j$ are independent of $X$.
Before proving Theorem \ref{thm:depevals}
 we shall introduce some terminology
pursuant to our more general result.

\subsubsection{Direct Sums}
\label{sec:dirSum} We present some definitions about direct sum
and sets which respect direct sums, since they are important
tools.

\begin{definition}
Our definition of the \df{direct sum}  is the usual one.
Given
pairs $(X_1,v_1)$ and $(X_2,v_2)$ where
     $X_j$ are $n_j\times n_j$ matrices and $v_j \in \RR^{n_j}$,
$$
(X_1,v_1)\oplus (X_2,v_2)= (X_1 \oplus X_2, v_1 \oplus v_2)
$$
where
$$X_1 \oplus
X_2:=\bmat 
X_1&0\\ 0&X_2\emat
\quad \quad
v_1 \oplus v_2 :=
\bmat
v_1 \\ v_2\emat.
$$
We extend this definition to $\mu$ terms,
$(X_1,v_1),\dots,(X_\mu,v_\mu)$ in the expected way.
\end{definition}

In the definition below, we consider a set $\cB$ which is the sequence
$$
\cB:= (\cB(n)), 
$$
where each $\cB(n)$ is a set whose members
are pairs $(X,v)$ where $X$ is in $\gtupn$
    and $v \in \RR^n$.

\begin{definition}
     The {\it set $\cB$   is said to respect direct sums}\index{direct sum!respects}
     if $(X^j, v^j)$
     with $X^j \in (\SS^{n_j\times n_j})^g$
      and $v^j\in \RR^{n_j}$
     for   $j=1,\dots, \mu$ being contained in the
     set $\cB(n_j)$ implies that the direct sum
$$
     (X^1 \oplus  \ldots \oplus X^\mu ,
     v^1 \oplus  \ldots \oplus v^\mu)
      =(\oplus_{j=1}^\mu  X^j, \oplus_{j=1}^\mu v^j)
$$
      is also contained in  $\cB(\sum n_j)$.
     \end{definition}

\begin{definition}
     By a \df{natural map} \rm $G$ on $\cB$,
     we mean a sequence of
     functions $G(n):\cB(n) \to \RR^n$,
     which \df{respects direct sums}
     in the sense that,
     if $(X^j,v^j)\in\cB({n_j})$ for $j=1,2,\dots,\mu$,
     then
     $$
     G(\sum_1^\mu n_j) (\oplus X^j,\oplus v^j)
     = \oplus_1^\mu G(n_j)(X^j,v^j).
     $$
Typically we omit the argument $n$,
writing $G(X)$ instead of $G(n)(X)$.
\end{definition}

Examples of sets which respect direct sums and
of natural maps are provided
by the following example.

\bexa\label{ex:respDS}
Let $\rho$ be a rational expression.
 \begin{enumerate}[\rm (1)]
\item
  The set $\cB^\rho=\{(X,v):
  X\in \cD^\rho \cap \gtupn
, \ v\in \RR^n ,\ n\in\NN \}$
  respects direct sums.
\item
 If $G$ is a matrix-valued nc rational expression whose domain
contains $\cD^\rho$, then $G$ determines a natural
map on $\cB(\rho)$ by $G(n)(X,v)=G(X)v$.
 In particular, every nc polynomial determines a natural
  map on every  nc basic  semialgebraic set $\cB$.
\end{enumerate}
\eexa

\subsubsection{Main Result on Linear  Dependence}

\begin{theorem}
     \label{thm:gendep}
      Suppose $\cB$ is a  set  which respects direct sums and
      $G_1,\ldots,G_\ell$ are natural maps on $\cB$.
      If for each $(X,v)\in\cB$ the
set $\{G_1(X,v),\ldots, G_\ell(X,v)\}$ is linearly
      dependent, then there exists  a nonzero
      $\lambda \in \RR^{\ell}$ so that
\bes
         0=
      \sum\limits^{\ell}_{j=1}
\lambda_j G_j(X,v)
\ees
     for every $(X,v) \in  \cB $.
     We emphasize that $\lambda$
     is independent of $(X,v)$.
\end{theorem}
\index{rational function!linear dependence}

Before proving \ref{thm:gendep}, we use it to prove an important earlier theorem.

\begin{proof}[Proof of Theorem  {\rm\ref{thm:depevals}}]
     Let
     $\cB$ be given by
     $$
       \cB(n) =
       \{(X,v)\colon X \in \cD^\rho \cap \gtupn \text{ and } v \in \RR^n \}.
     $$
     Let $G_j$ denote the natural maps, $G_j(X,v)=G_j(X)v$.
     Then $\cB$ and $G_1,\dots,G_\ell$
     satisfy the hypothesis of Theorem \ref{thm:gendep} and so the first
conclusion of Theorem \ref{thm:depevals} follows. 

The last conclusion follows because an nc rational function
$r$
vanishing on an nc basic open semialgebraic set  is 0 on all $\dom(r)$ and
hence is zero, cf.~Exercise \ref{exe:ratVanish}.
\end{proof}

\subsubsection{Proof of Theorem {\rm\ref{thm:gendep}}}
   We start with a finitary version of Theorem \ref{thm:gendep}:

\begin{lem}
     \label{lem:genfinitedep}
     Let $\cB$ and $G_i$ be as in 
Theorem {\rm\ref{thm:gendep}}.
     If $\cR$ is a finite subset of $\cB$, then
     there exists a nonzero
     $\lambda({\cR})\in \RR^{\ell}$
     such that
\bes
      \sum\limits^{\ell}_{j=1}
       \lambda({\cR})_j G_j(X)v=0,
\ees
for every $(X,v)\in {\cR}$.
\end{lem}

\begin{proof}
     The proof relies on taking direct sums of matrices.  Write the set
     $\cR$ as $${\cR}= \big\{(X^1, v^1), \dots ,
     (X^\mu, v^\mu) \big\},$$
     where each $(X^i, v^i) \in \cB$. Since
     $\cB$ respects direct sums,
     $$
       (X,v)=(\oplus_{\nu=1}^\mu  X^\nu,\oplus_{\nu=1}^\mu v^\nu)\in\cB.
     $$
     Hence, there exists  a nonzero
     $\lambda({\cR})\in \RR^{\ell}$  such that
\bes
          0=  \sum\limits^{\ell}_{j=1}
            \lambda({\cR})_j G_j(X,v).
\ees
     Since each $G_j$ respects direct sums, 
the desired conclusion follows.
\end{proof}

\begin{proof}[Proof of Theorem {\rm\ref{thm:gendep}}]
The proof is essentially a compactness argument,
based on Lemma \ref{lem:genfinitedep}.
Let $\BB$ denote the unit sphere in $\RR^\ell$.

     To  $(X,v)\in \cB$ associate the set
     $$
       \Omega_{(X,v)}=\big\{\lambda \in \BB \colon \lambda\cdot G(X)v
= \sum_j \lambda_j G_j(X,v)=0\big\}.
     $$
     Since $(X,v)\in\cB,$ the hypothesis on $\cB$
     says $\Omega_{(X,v)}$ 
is nonempty.
     It is evident that $\Omega_{(X,v)}$ is a
     closed subset of $\BB$ and is thus compact.

Let ${\bf \Omega}:=\{\Omega_{(X,v)}\colon (X,v)\in\cB\}$.
Any finite sub-collection from ${\bf \Omega}$ has the form
$\{ \Omega_{(X,v)} \colon (X,v) \in \cR  \}$
for some finite subset $\cR$ of $\cB$,
and so by Lemma
     \ref{lem:genfinitedep}
     has a nonempty intersection.
     In other words, ${\bf \Omega}$ has the finite intersection property.
The compactness of $\BB$ implies  that there is a
$\lambda \in \BB$ which is
     in every $\Omega_{(X,v)}$.
     This is the desired conclusion of the theorem.
\end{proof}

\subsection{Proof of the QuadratischePositivstellensatz}
We are now ready to give the proof of Theorem \ref{thm:qposSS}.
 Accordingly, let $\cO$ be a given basic open semialgebraic set.\index{QuadratischePositivstellensatz}
Suppose 
\beq \label{eq:qVZV} 
      q(x)[h]= V(x)[h]^\ss \, Z(x)\, V(x)[h], \eeq
where $V$ is the border vector and $Z$ is the middle matrix;
cf.~\eqref{eq:rep1}.
Clearly, if $Z$ is matrix-positive on $\cO$, 
 then $q(X)[H]$ is positive semidefinite for each $n,$ 
 each $X\in\cO(n)$ and $H\in\gtupn$.

The converse is less trivial and requires the CHSY Lemma
plus our main result on linear dependence of nc rational functions.
Let $\ell$ denote the degree of $q(x)[h]$ in the variable $x$.
 In particular, the border vector in the representation 
 of $q(x)[h]$ itself has degree $\ell$ in $x$. 
 Recall $\sigma_\ell$ from Exercise \ref{exe:fock}.

 Suppose for some $s$ and $g$-tuple of symmetric matrices 
 $\tX=(\tX_1,\dots,\tX_g)\in\cO(s)$,
 the matrix $Z(\tX)$ is not positive semidefinite.
 By Lemma \ref{lem:CHSY-in-action} and Theorem \ref{thm:depevals}, there
 is an $t,$ a $Y\in\cO(t)$, and a vector $\eta$ so that
 $\{m(Y)\eta \colon m\in\ax_\ell\}$ is linearly independent.
 Let $X=\tX \oplus Y$ and $\gamma=0\oplus \eta\in\RR^{s+t}$.  Then
 $Z(X)$ is not positive semidefinite and 
 $\{m(X)\gamma \colon m\in\ax_\ell\}$ is linearly independent. 

 Let $N=g\frac{\kappa (\kappa-1)}{2}+1$, where $\kappa$
  is given in Lemma \ref{lem:CHSY-in-action} and
  let $n=(s+t)N$. 
 Consider
$W=X\otimes I_N=(X_1\otimes I_N,\ldots, X_g\otimes I_N)$
 and vector $\omega = \gamma\otimes e$, for 
 any nonzero vector $e\in\RR^{N+1}$. 
  The set $\{m(W)\omega \colon m\in \ax_\ell\}$ 
 is linearly independent and thus by Lemma \ref{lem:CHSY-in-action},
  the codimension of $\cM=\{V(W)[H]\omega \colon H\in\gtupn\}$
  is at most $N-1$.  On the other hand, because
  $Z(X)$ has a negative eigenvalue, 
 the matrix
$Z(W)$ has an eigenspace $\cE$, corresponding to a negative
eigenvalue, of dimension at least $N$.  It follows that
  $\cE\cap \cM$ is nonempty; i.e., there is an $H\in\gtupn$
 such that $V(W)[H]\omega \in \cE$.  In  particular, this together with \eqref{eq:qVZV} implies
\bes
 \langle q(W)[H]\omega,\omega\rangle
  = \langle Z(W)V(W)[H]\omega,V(W)\omega \rangle <0
\ees
 and thus, $q(W)[H]$ is not
 positive semidefinite. 
\qed

\subsection{Exercises}

\bexe\label{exe:lemCHSY}
Prove Lemma {\rm\ref{lem:CHSY}.}
\eexe

\bexe
 \label{exe:Apower}
Let $A\in \RR^{n\times n}$ be given.  Show,  if the rank
 of $A$ is $r$, then 
 the matrices
$A,A^2,\ldots, A^{r+1}$ are linearly dependent.
\eexe

In the next exercise employ the Fock space\index{Fock space}
(see Section \ref{sec:2exes}) to prove 
a strengthening
of
Corollary \ref{cor:depevalsSimple0}
for nc polynomials.

\bexe\label{exe:fockRules}
Suppose $p_1,\ldots,p_\ell\in\RR\ax_k$ are nc polynomials.
 Show, if the set of vectors 
\beq\label{eq:fockRules}
\{p_1(X)v,\ldots,p_\ell(X)v\}
\eeq
 is linearly dependent
for every $(X,v)\in (\SS^{\sigma\times\sigma})^g\times \RR^\sigma$,
where $\sigma=\sigma(k)=\dim\RR\ax_k$, then
 $\{p_1,\ldots,p_\ell\}$ is  linearly dependent.
\eexe
\index{polynomial!linear dependence}

\bexe
Redo Exercise {\rm\ref{exe:fockRules}} under the assumption that
the vectors \eqref{eq:fockRules} are linearly dependent
for all 
$(X,v)\in O\times\RR^\sigma$, where $O\subseteq(\SS^{\sigma\times\sigma})^g
$ is a nonempty open set.
\eexe

For a more algebraic view of the linear dependence of
nc polynomials we refer to \cite{BK}.

\bexe
Prove that $f\in\RR\ax$ is a sum of squares if and only if it has a positive
semidefinite Gram matrix. Are then all of $f$'s Gram matrices positive
semidefinite?
\eexe\index{sum of squares}\index{Gram matrix}

\section{NC varieties with positive curvature have degree two}
 \label{sec:variety}

This section looks at \nc \ varieties and their geometric properties.
We see a very strong rigidity when  they have positive curvature
which generalizes what we have already seen about 
convex polynomials (their graph is a positively curved variety)
 having degree two. 

In the classical setting of a surface defined by the zero set
$$\nu(p)=\{x\in\RR^g \colon p(x)=0\}$$ of a
polynomial $p=p(x_1,\dots,x_g)$ in $g$ commuting variables,
the second fundamental form\index{second fundamental form} at a smooth point
$x_0$ of $\nu(p)$ is the quadratic  form,
\begin{equation}
\label{def:2ndfform}
h\mapsto 
  -\langle (\text{Hess}\,p)(x_0) h, h\rangle,
\end{equation}
 where $\text{Hess}\,p$ is the Hessian\index{Hessian} of $p$, and $h\in\RR^g$ is
 in the tangent space to the surface $\nu(p)$
 at $x_0$; i.e., $\nabla p(x_0)\cdot h=0$.\footnote{The choice of the minus sign in \eqref{def:2ndfform} is somewhat arbitrary.
Classically the sign of the second fundamental form is associated with the
choice of a smoothly varying vector that is normal to $\nu(p)$.
The zero set $\nu(p)$ has positive curvature\index{positive curvature}
 at $x_0$ if the second fundamental form is
either positive semidefinite or  negative semidefinite at $x_0$.
For example,
if we define $\nu(p)$ using a concave function $p$,
then the second fundamental form is
negative semidefinite, while
for the same set $\nu(-p)$ the second fundamental form is
positive  semidefinite.
}

 We shall show that in the noncommutative setting
the zero set $\cV(p)$ of a noncommutative
 polynomial $p$ (subject to appropriate irreducibility constraints)
 having positive curvature (even in a small neighborhood) 
implies that $p$ is convex - and thus, $p$ 
 has degree at most two - and $\cV(p)$ has
 positive curvature everywhere; see
 Theorem \ref{thm:sigmain}
 for the precise statements.

In fact there is 
  a natural
  notion of the signature $C_\pm(\cV(p))$ of a variety $\cV(p)$ 
   and the bound 
$$ \deg( p)  \leq  2 C_\pm (\cV(p) ) + 2$$
on the degree of $p$ in terms of the signature $C_\pm (\cV(p))$
was obtained in \cite{DHM07b}.
  The convention that     $C_+(\cV(p))=0$  corresponds 
  to positive curvature,
 since in our examples, defining functions $p$ are
typically concave or  quasiconcave.
 One could consider characterizing $p$ for which 
  $C_\pm(\cV(p))$ satisfies less restrictive hypothesis
  than equal zero and this has been done to some extent 
  in \cite{DGHM}; however, 
  this higher  level of generality is beyond our focus here.
Since our goal is to present the basic ideas,
we stick to positive curvature.

\subsection{NC varieties and their curvature}

We next define a number of basic geometric objects associated to the nc
variety determined by an nc polynomial $p$.

\subsubsection{Varieties, tangent planes, and the second fundamental form}
 \label{subsec:ncsag}

The \df{variety} (zero set)\index{variety!nc}\index{zero set} of a  $p\in\RR\ax$ is 
$$
\cV(p):=\bigcup_{n\ge1}\mathcal{V}_n(p),
$$
where
$$
\mathcal{V}_n(p):=\left\{(X,v)\in\gtupn\times\RR^n\colon p(X)v=0\right\}.
$$
The \df{clamped tangent plane} to $\cV(p)$ at
\index{tangent plane!clamped}
$(X,v)\in\mathcal{V}_n(p)$ is
\begin{equation*}
  \ttpp{p}{X}{v}:= \{ H \in\gtupn\colon  p^\prime(X)[H]v=0 \}.
\end{equation*}
The \df{clamped second fundamental form}\index{second fundamental form!clamped}
   for $\cV(p)$
  at  $(X,v)\in\mathcal{V}_n(p)$ is the quadratic form
\begin{equation*}
  \ttpp{p}{X}{v}\to\RR,\quad H \mapsto
   - \langle p^{\prime\prime}(X)[H]v,v \rangle.
\end{equation*}
Note that
$$
\{ X \in \gtupn \colon  (X,v) \in \cV(p) \ \textrm{for some} \
v\ne 0 \}
  =\{X\in\gtupn\colon  \det(p(X))=0\}
$$
is a variety in $\gtupn$ and typically has a  {\it true}
(commutative) tangent plane
at many points $X$,
which  of course  has codimension one, whereas
the clamped tangent plane at a typical point
$(X, v)\in \cV_n(p)$ has codimension on the order of $n$ and
is contained inside the true tangent plane.

\subsubsection{Full rank points}
 \label{subsec:smooth}

  The point $(X,v)\in \cV(p)$ is a \df{full rank point}
  of $p$ if the mapping
$$
\gtupn\to\RR^n, \quad 
H \mapsto p^\prime(X)[H]v 
$$
  is onto. 
  The full rank condition is   a nonsingularity condition
  which amounts
  to a smoothness hypothesis. Such 
  conditions play a major role in real algebraic geometry,
  see  \cite[\S 3.3]{BCR98}.

As an example, consider the classical real
algebraic geometry case of $n=1$
 (and thus $X\in\RR^g$)
with the commutative polynomial $\check{p}$ (which can be taken 
to be the \df{commutative collapse} of the polynomial $p$).
In this case, a full rank point $(X, 1)\in\RR^g\times\RR$
is a point at which the gradient of $\check{p}$
does not vanish.
Thus, $X$ is a nonsingular point for the zero variety of $\check{p}$.

Some  perspective for $n>1$ is obtained by counting dimensions.
If $(X, v)\in \gtupn\times\RR^n$, then $H\mapsto p^\prime(X)[H]v$
is a linear map
from the $g(n^2+n)/2$ dimensional space $\gtupn$ into
the $n$ dimensional space $\RR^n$.
Therefore, the codimension of
the kernel of this map is no bigger than $n$.
This codimension is $n$ if and only  if  $(X,v)$ is a full rank point
and in this case the clamped tangent plane has codimension $n$.

\subsubsection{Positive curvature}
 \label{subsec:curvature}
  As noted earlier, a notion of positive (really nonnegative)
  curvature can be defined in terms of
  the clamped second fundamental form.

The variety $\cV(p)$ has \df{positive curvature}
at $(X,v)\in \cV(p)$
  if the clamped second fundamental form is nonnegative at $(X,v)$; i.e.,
  if
$$
\csff{p}{X}{H}{v} \ge 0\quad\textrm{ for every}\quad H\in\ttpp{p}{X}{v}\,.
$$

 \subsubsection{Irreducibility: The minimum degree defining polynomial
condition}
 \label{subsec:irreducible}

  While there is no tradition of what is an effective notion
  of  irreducibility for nc polynomials,
  there is a notion of minimal degree nc polynomial
  which is appropriate for the present context.
  In the commutative case
  the polynomial $\check p$ on $\RR^g$ is
  a \df{minimal degree defining polynomial} for $\nu(\check p)$
  if there does not exist
  a polynomial $q$ of lower degree  such that
  $\nu(\check p)= \nu(q)$.
  This is a key feature of irreducible polynomials.\index{polynomial!irreducible}

\begin{definition}\label{it:mindegKey}
  A symmetric nc polynomial $p$ is a
  \df{minimum degree defining polynomial} for
a nonempty set 
$\cD \subseteq \cV(p)$
 if
 whenever $q\ne 0$ is another {\it $($not necessarily symmetric$)$}
 nc   polynomial such that $q(X)v=0$ for each $(X,v)\in\cD$,
    then
   $$
\deg (q) \ge \deg(p).
    $$
Note this contrasts with {\rm\cite{DHM07a}},
where minimal degree meant a slightly weaker inequality holds.
\end{definition}

The reader who is so
  inclined can simply choose $\cD=\cV(p)$  or
  $\cD$ equal to the full rank points of $\cV(p)$.

Now we give an example to illustrate these ideas.

\subsection{A very simple example}
 In the following example, the
 null space
\begin{equation*}
  \cT=\cT_p(X,v)=\{H\in\gtupn \colon p^\prime(X)[H]v=0\}
\end{equation*}
 is computed for certain choices of $p$, $X$, and $v$.
 Recall that if $p(X)v=0$, then the
  subspace $\cT$ is the {\it clamped tangent plane}
  \index{clamped tangent plane} introduced in
  Subsection \ref{subsec:ncsag}.

\begin{example}\rm
 \label{ex:simple}
Let $X\in\Snn$, $v\in\RR^n$, $v\ne 0$, let $p(x)=x^k$ for some
integer $k\ge 1$.
Suppose that $(X,v) \in \cV(p)$,
that is,  $X^kv=0$. Then, since
$$
X^kv=0\Longleftrightarrow Xv=0 \quad\textrm{when}\ X\in\Snn,
$$
it follows that $p$ is a minimum degree defining polynomial
for $\cV(p)$ if and only if $k=1$.

It is readily checked that
$$
(X,v) \in \cV(p)\Longrightarrow p^\prime(X)[H]v=X^{k-1}Hv,
$$
and hence that $X$ is a full rank point for $p$ if and only
 if $X$ is invertible.

Now suppose $k\ge 2$. Then,
$$
\langle p^{\prime\prime}(X)[H]v , v\rangle=2\langle HX^{k-2}Hv, v\rangle.
$$
Therefore, if $k>2$
$$
(X,v) \in \cV(p)  \quad \textrm{and} \quad p^\prime(X)[H]v=0
\ \ \Longrightarrow \  \ X Hv=0, \ \textrm{and so}
$$
$$
\langle p^{\prime\prime}(X)[H]v , v\rangle=0.
$$

To count the dimension of $\cT$  we can suppose without loss of generality
that
$$
X=\begin{bmatrix}0&0\\ 0&Y\end{bmatrix} \quad\text{and}\quad
v=\begin{bmatrix}1& 0& \cdots& 0\end{bmatrix}^\ss,
$$
where $Y\in\SR^{(n-1)\times (n-1)}$ is invertible.
Then, for the simple case under consideration,
$$
\cT=\{ H\in \SR^{\ntn}\colon h_{21},\ldots, h_{n1}=0\},
$$
where $h_{ij}$ denotes the $ij$ entry of $H$. Thus,
$$
\textup{dim}\,\cT=\frac{n^2+n}{2}-(n-1),
$$
i.e.,
$\textup{codim}\,\cT=n-1.$
\end{example}

\begin{remark}\rm
We remark that
$$
X^k v=0\quad\textrm{and}\quad\langle p^{\prime\prime}(X)[H]v , v\rangle=0
\Longrightarrow p^\prime(X)[H]v=0\quad\textrm{if}\ k=2t\ge 4,
$$
as follows easily from the formula
$$
\langle p^{\prime\prime}(X)[H]v , v\rangle= 2\langle X^{t-1}Hv,
X^{t-1}Hv\rangle.
$$
\end{remark}

\bexe
Let $A\in\SRnn$ and let $\cU$ be a maximal strictly negative
subspace of $\RR^n$ with respect to the quadratic form
$\langle Au, u\rangle$. Prove: there exists a complementary subspace $\cV$ of
$\cU$ in $\RR^n$ such that $\langle Av, v\rangle\ge 0$ for every $v\in\cV$.
\eexe

\subsection{Main Result:  Positive curvature and the degree of $p$}
 \label{subsec:results1}

    \begin{theorem}
    \label{thm:sigmain}
     Let $p$ be a symmetric nc polynomial in $g$ symmetric variables,
      let $\cO$ be a nc  basic  open  semialgebraic set
      and let $\cR$ denote the full rank points of $p$
      in $\cV(p)\cap \cO.$
     If
     \ben[\rm (1)]
        \item $\cR$ is nonempty; 
        \item  $\cV(p)$
         has positive curvature at each point of $\cR$;
          and
        \item $p$ is a minimum degree defining polynomial
          for $\cR$,
\een
         then $\deg(p)$ is at most two and $p$ is concave.
 \end{theorem}
\index{positive curvature}

\subsection{Ideas and proofs}

Our aim is to give the idea behind the proof of Theorem \ref{thm:sigmain}
under much stronger hypotheses.
We saw earlier
the positivity of a quadratic on a nc basic open  set $\cO$ 
imparts positivity to its MM there.
The following shows this happens for thin sets (nc varieties) too.
 Thus, the following theorem generalizes 
 the QuadratischePositivstellensatz, Theorem \ref{thm:qposSS}.
\index{QuadratischePositivstellensatz}

 \begin{theorem}
    \label{lem:sigmainWeak}
Let $p,\cO,\cR$ be as in Theorem {\rm\ref{thm:sigmain}}.
   Let $q(x)[h]$ be a polynomial which is
    quadratic in $h$ having MM representation
   $ q= V^{\ss} Z V$
   for which    $\deg  (V) \leq  \deg ( p)$.
   If 
   \beq
   \label{eq:fpos}
   v^{\ss} q(X)[H] v \geq 0 \quad \text{for  all} \quad (X,v) \in \cR
\text{ and all }H,
   \eeq 
   then
   $Z(X)$ is positive semidefinite for all $X$ with $(X,v) \in \cR$.      
 \end{theorem}

\def\hX{\hat X}
\def\hv{\hat v}

\begin{proof}
The proof of this theorem follows the 
proof of the {\rm\qpos}, modified
 to take into account  the set $\cR$. 

Suppose for each $(X,v)\in\cR$ there is a linear
combination $G_{(X,v)}(x)$ of the 
words $\{m(x) \colon \deg(m)<\deg(p)\}$ with $G_{(X,v)}(X)v=0$ for
all $(X,v)\in\cR$. Then
by Theorem \ref{thm:gendep} (note that $\cR$ is closed under
direct sums), there is a linear combination
$G\in\RR\ax_{\deg(p)-1}$ with $G(X)v=0$. However, this is absurd
by the minimality of $p$.
Hence there is an $(Y,v)\in\cR$ such that $\{m(Y)v\colon \deg(m)<\deg(p)\}$
is linearly independent.

 Assume for some $g$-tuple of symmetric matrices 
 $\tX=(\tX_1,\dots,\tX_g)$,
 there is a vector $\tilde v$ such that $(\tX,\tilde v)\in\cR$,
 and the matrix $Z(\tX)$ is not positive semidefinite.
 Let $X=\tX \oplus Y$ and $\gamma=\tilde v\oplus v$.  Then
 $(X,\gamma)\in\cR(\ell)$ for some $\ell$;  the matrix
 $Z(X)$ is not positive semidefinite; and 
 $\{m(X)\gamma \colon \deg(m)<\deg(p)\}$ is linearly independent. 

 Let $N=g\frac{\kappa (\kappa-1)}{2}+1$, where $\kappa$
  is given in Lemma \ref{lem:CHSY-in-action} and
  let $n=\ell N$. 
 Consider
$W=X\otimes I_N=(X_1\otimes I_N,\ldots, X_g\otimes I_N)$
 and vector $\omega = \gamma\otimes e$, 
 where $e\in\RR^N$ is the vector with each entry 
 equal to $1.$  Then, $(W,\omega)\in\cR(n)$, 
  and the set $\{m(W)\omega \colon m\in \ax_\ell\}$ 
 is linearly independent and thus by Lemma \ref{lem:CHSY-in-action},
  the codimension of $\cM=\{V(W)[H]\omega \colon H\in\gtupn\}$
  is at most $N-1$.  On the other hand, because
  $Z(X)$ has a negative eigenvalue, 
 the matrix
$Z(W)$ has an eigenspace $\cE$, corresponding to a negative
eigenvalue, of dimension at least $N$.  It follows that
  $\cE\cap \cM$ is nonempty; i.e., there is an $H\in\gtupn$
 such that $V(W)[H]\omega \in \cE$.  In  particular,
\[
 \langle q(W)[H]\omega,\omega\rangle
  = \langle Z(W)V(W)[H]\omega,V(W)\omega \rangle <0
\] 
 and thus, $q(W)[H]$ is not
 positive semidefinite. 
\end{proof}

\subsubsection{The modified Hessian}

 Our main tool for analyzing the curvature of noncommutative varieties
 is a variant of the Hessian for symmetric nc polynomials $p$.
The
curvature of $\cV(p)$ is defined in terms of $\textup{Hess}\,(p)$
compressed to
tangent planes, for each dimension $n$.
This compression of the Hessian is awkward to work with directly, and
so we associate to it a quadratic polynomial $q(x)[h]$ 
carrying all of the information of $p''$ compressed to the tangent plane,
but having the key property \eqref{eq:fpos}.
We shall call this $q$ we construct the relaxed Hessian.
 The first step in constructing the relaxed Hessian is
  to consider the simpler 
 \df{modified Hessian}\index{Hessian!modified}
\bes
   p^{\prime\prime}_{\lambda, 0}(x)[h]:=
      p^{\prime\prime}(x)[h]       
        +\lambda \,  p^{\prime}(x)[h]^\ss p^\prime(x)[h].
\ees
  which captures
  the conceptual idea. 
  Suppose $X\in \gtupn$ and $v\in \RR^n$.
  We say that the {\it modified Hessian is negative} at
$(X,v)$ if there is a $\lambda_0<0$, so that
 for all $\lambda \le \lambda_0$,
$$
   0\le
    -\langle p^{\prime\prime}_{\lambda, 0}(X)[H]v,v\rangle
$$
 for all $H\in\gtupn$.
 Given a
  subset $\cR=(\cR(n))_{n=1}^\infty$, with
$\cR(n)\subseteq \gtupn \times \RR^n $,
 we say that the
{\it modified Hessian is  negative on $\cR$}
  if
 it is  negative at each $(X,v)\in S$.

Now we turn to motivation.

\begin{example} \rm
 \label{ex:classical}
{\it The classical $n=1$ case.}
Suppose that $p$ is strictly smoothly quasi-concave, meaning that
all superlevel sets of $p$ are strictly convex with
strictly positively curved smooth boundary.
Suppose that the gradient $\nabla p$ (written as a row vector) never vanishes on  $\RR^g$. 
Then 
$G = \nabla p (\nabla p)^\ss$ is strictly positive, 
 at each point $X$ in $\RR^g$. 
Fix such an $X$;
 the modified Hessian  can be decomposed as a block matrix
subordinate to the tangent plane
to the level set at $X$, denoted $T_X$,
and to its orthogonal complement (the gradient direction):
$$ T_X \oplus \{ \lambda  \nabla p\colon \lambda \in \RR \}.$$
In this decomposition the modified Hessian 
has the form
$$ R= \bmat
  A & B \\
  B^\ss & D+ \lambda G \\
\emat.
$$
Here, in the case of $\lambda=0$, $R$ 
is the Hessian and
the second fundamental form is $A$
or $-A$, depending on convention and 
 the rather arbitrary choice of inward
or outward normal to $\nu$.
 If we select our  normal direction to be $\nabla p$,
then $-A$ is the classical second fundamental form
as is consistent with the choice of sign in our definition in Subsection 
\ref{subsec:curvature}. 
(All this concern with the sign is unimportant  to the content 
of this chapter
and can be ignored by the reader.)

Next, in view of the presumed strict positive
curvature  of each level set $\nu$, the matrix $A$
at each point of $\nu$ is  negative definite
but the Hessian could have a negative eigenvalue.
However, by standard Schur complement arguments,
$R$ will be negative
definite 
if
$$
D+\lambda G-B^\ss A^{-1}B\prec 0
$$
on this region.
Thus, strict convexity assumptions on the sublevel sets of
$p$ make the modified Hessian negative
definite for negative enough $\lambda$.
One can make this negative definiteness
uniform in $X$ in various neighborhoods
under modest assumptions.
\end{example}

Very unfortunately
in the noncommutative case,
Remark 6.8 \cite{DHMppt}
 implies that 
if $n$ is large enough, then the second fundamental form
will have a nonzero null space,
thus strict negative definiteness of the 
$A$ part of the modified Hessian
is impossible.

Our trick, to deal with the likely reality that $A$ is 
only positive semidefinite, and 
obtain a negative definite $R$, is to add
another negative term, say $\delta I$, with arbitrarily small
$\delta < 0$.
After adding such $\delta$, the argument based on
choosing $-\lambda$ large succeeds as before. 
This $\delta$ term plus the $\lambda$
term produces the ``relaxed Hessian", to be introduced next,
and proper selection of these terms make it negative definite.

 \subsubsection{The relaxed Hessian}

Recall\index{relaxed Hessian}\index{Hessian!relaxed}
  Let $V_k(x)[h]$ denotes the vector of polynomials with entries
$h_jw(x)$, where $w\in\ax$ runs through the set of $g^k$ words of length
$k$,
$j=1,\ldots,g$. Although the order of the entries is fixed in some of our
earlier applications (see e.g.~\cite[(2.3)]{DHM07b}) it is
irrelevant for
the
moment.
Thus, $V_k=V_k(x)[h]$ is a vector of height $g^{k+1}$, and the vectors
\bes
V(x)[h]=\textup{col}(V_0,\ldots,V_{d-2})\quad \textrm{and}\quad
\widetilde{V}(x)[h]=\textup{col}(V_0,\ldots,V_{d-1})
\ees
are vectors of height $g\sigma(d-2)$ and $g\sigma(d-1)$ respectively.
Note that
$$
  \widetilde{V}(x)[h]^\ss \widetilde{V}(x)[h] = \sum_{j=1}^g \
  \sum_{\deg(w)\le d-1} \; w(x)^\ss h_j^2 w(x).
$$
 
 The \df{relaxed Hessian}\index{Hessian!relaxed} 
 of the symmetric nc polynomial
 $p$ of degree $d$ is defined to be 
\bes
   p^{\prime\prime}_{\lambda, \delta}(x)[h]:=  
   p^{\prime\prime}_{\lambda, 0}(x)[h] 
     + \delta \, \widetilde{V}(x)[h]^\ss \widetilde{V}(x)[h]
\in\RR\ax[h].
\ees
  Suppose $X\in \gtupn$ and $v\in \RR^n$.
 We say that the {\it relaxed Hessian is negative}
  at
$(X,v)$ if
 for each $\delta<0$ there is a $\lambda_\delta<0$, so that
 for all $\lambda \le \lambda_\delta$,
$$
   0\le
    -\langle p^{\prime\prime}_{\lambda, \delta}(X)[H]v,v\rangle
$$
 for all $H\in\gtupn$.
 Given a 
  $\cR=(\cR(n))_{n=1}^\infty$, with
$\cR(n)\subseteq \gtupn \times \RR^n $,
 we say that the
{\it relaxed Hessian is positive $($resp., negative$)$ on $\cR$}
  if
 it is positive (resp., negative) at each $(X,v)\in S$.

 The following theorem provides a link
 between the signature of the clamped
 second fundamental form with that of the
 relaxed Hessian.

\begin{theorem}
 \label{thm:signature-clamped-relaxed}
    Suppose $p$ is a symmetric nc polynomial of degree $d$ in $g$
    symmetric variables and $(X,v)\in
\gtupn\times\RR^n$.
  If $\cV(p)$ has  positive  curvature at
   $(X,v)\in \cV_n(p)$,
   i.e., if
$$
\langle p^{\prime\prime}(X)[H]v,v\rangle\le 0\quad\text{for every}\ H\in
\cT_p(X,v),
$$
  then for every $\delta<0$
   there exists a
$\lambda_\delta<0$ such that for all $\lambda \le \lambda_\delta$,
$$
\langle p_{\lambda,\delta}^{\prime\prime}(X)[H]v,v\rangle\le 0
\quad\text{for every}\ H\in \gtupn;
$$
i.e.,  the relaxed Hessian of $p$ is negative at $(X,v)$.
\end{theorem}
We leave the proof of Theorem \ref{thm:signature-clamped-relaxed}
  to the reader.

The basic idea of the proof of
 Theorem \ref{thm:sigmain}, 
 is to obtain a negative relaxed Hessian $q$ from
  Theorem \ref{thm:signature-clamped-relaxed} and then 
   apply Theorem \ref{lem:sigmainWeak}.  We begin
 with the following lemma.

\begin{lem}
 \label{easy-lemma-variant}
   Suppose $R$ and $T$ are operators on a finite
   dimensional Hilbert space $H=K\oplus L$.  
   Suppose further that, with respect to this decomposition of 
  $H$,  the operator  $R=CC^\ss$ for
\[
   C=\begin{bmatrix} r \\ c \end{bmatrix} : L \to K\oplus L \quad\text{and}\quad 
  T=\begin{bmatrix}  T_0 & 0 \\ 0 & 0 \end{bmatrix}. 
 \]
   If $c$ is invertible and if for every $\delta>0$ there is 
  a $\eta>0$ such that for all $\lambda>\eta$, 
\[
  T+\delta I +\lambda R \succeq 0,
\]
 then $T\succeq 0$.
\end{lem}

\begin{proof}
  Write
\[
  T+\delta I +\lambda R 
   =\begin{bmatrix} T_0 + \delta I+ \lambda  rr^\ss & \lambda rc^\ss \\
            \lambda cr^\ss & \delta + \lambda cc^\ss \end{bmatrix}.
\]
From Schur complements
 it follows that
\[
  T_0 +\delta I + r(\lambda - \lambda^2 c^\ss(\delta+\lambda cc^\ss)^{-1} c )r^\ss
  \succeq 0.
\]
 Now
\[
\begin{split}
r(\lambda - \lambda^2 c^\ss(\delta+\lambda cc^\ss)^{-1} c )r^\ss
 & =   \lambda 
    rc^\ss((c c^\ss)^{-1} - \lambda (\delta+ \lambda cc^\ss)^{-1})
 c r^\ss \\
 & =  \lambda rc^\ss \delta (cc^\ss)^{-1}(\delta +\lambda (cc^\ss))^{-1} c r^\ss \\
 & \preceq \delta r (cc^\ss)^{-1} r^\ss.
 \end{split}
\]
 Hence,
\[
 T_0+\delta I + \delta r(cc^\ss)^{-1} r^\ss  \succeq  0.
\]
 Since the above inequality holds for all $\delta>0$, it follows
 that  $T_0 \succeq 0$.
\end{proof}

We now have enough machinery developed to prove
Theorem \ref{thm:sigmain}.

\begin{proof}[Proof of Theorem {\rm\ref{thm:sigmain}}]
  Fix $\lambda, \delta >0 $ and consider 
$q(x)[h]= -p_{\lambda,\delta}^{\prime\prime}(x)[h]$.
 We are led to investigate the middle matrix $Z^{\lambda,\delta}$ of
$q(x)[h]$,
 whose border vector $V(x)[h]$ includes all monomials of the
 form $h_j m$, where $m$ is a word in $x$ only of length at 
  most $d-1$; here $d$ is the degree of $p$.  Indeed,
\[
  Z^{\lambda,\delta}
    = Z +\delta I + \lambda W,
\]
where $Z$ is the middle matrix for $-p''(x)[h]$, and
$W$ is the middle matrix for the polynomial $p^\prime(x)[h]^\ss
 p^\prime(x)[h]$.  With an appropriate choice of ordering for the border 
vector $V$, we
 have, $W=CC^\ss$, where 
\[
   C(x) = \begin{bmatrix} w(x) \\ c \end{bmatrix},
\]
 for a nonzero vector $c$; 
 and at the same time,
\[
  Z(x) = \begin{bmatrix} Z^{0,0}(x) & 0 \\ 0 &  0\end{bmatrix}.
\]

By the curvature  hypothesis at a given 
 $X$ with $(X,v)\in\cR$, 
 Theorem \ref{thm:signature-clamped-relaxed}
 implies
   for every $\delta>0$ there is an $\eta>0$
  such that if $\lambda>\eta$
  $$\langle q(X)[H]v,v  \rangle\ge 0  \qquad \text{ for 
  all } (X,v)\in \cR \text{ and  all } H.$$ 
   Hence, by Theorem \ref{lem:sigmainWeak}, 
 the middle matrix, $Z^{\lambda,\delta}(X)$ for $q(x)[h]$
  is positive semidefinite. 
  We are in the setting of Lemma \ref{easy-lemma-variant} 
  from which we obtain
   $Z^{0,0}(X)\succeq 0$.
   If this held for $X$ in a nc basic open semialgebraic set, 
   then Theorem \ref{thm:convexd2}
    forces $p$ to have degree 
   no greater than 2. The proof of that theorem applies easily 
   here to finish this proof.
\end{proof}

\subsection{Exercises}

\bexe
Compute the BV-MM representation for the relaxed Hessian
of $x^3$ and $x^4$.
\eexe

\section{Convex semialgebraic nc sets}
 \label{sec:convexsemialg}

In this section we will give a brief overview of 
convex semialgebraic nc sets and positivity of nc polynomials on 
them. We shall see that their structure\index{semialgebraic set!convex}\index{polynomial!convex}
is much more rigid than that of their commutative counterparts.
For example, roughly speaking, each convex semialgebraic nc set
is a spectrahedron\index{spectrahedron}; i.e., a solution set of a linear matrix
inequality\index{linear matrix inequality}\index{LMI} 
(cf.~Subsection \ref{subsec:lmi} below). Similarly, every nc polynomial nonnegative on a
spectrahedron admits a sum of squares\index{sum of squares} representation with weights
and optimal degree bounds (see Subsection \ref{subsec:convPos} for details and
precise statements).

\subsection{nc Spectrahedra}\label{subsec:lmi}

Let $L$ be an affine linear pencil. Then the solution set of 
the linear matrix inequality (LMI) $L(x)\succ0$ is
\[
  \cD_L = \bigcup_{n\in\NN} \big\{X\in (\SS^{n\times n})^g \colon  L(X) \succ 0 \big\},
\]
and is called a \df{nc spectrahedron}.
 The set $\cD_L$ is convex in the
  sense that each 
\[
\cD_L(n):=\big\{X\in (\SS^{n\times n})^g \colon  L(X) \succ 0 \big\}
\]
 is convex.\index{convex}\index{nc basic open semialgebraic set!convex}
  It is also a noncommutative basic open semialgebraic set
as defined in Subsection \ref{subsec:ncsa} above.
 The main theorem of this section is the converse,
 a result
 which has implications for both semidefinite programming
 and systems engineering.

Most of the time we will focus on monic linear pencils.\index{linear pencil}\index{linear pencil!monic}
An affine linear pencil $L$ is called \df{monic} if $L(0)=I$, i.e.,
$L(x)= I + A_1 x_1 + \cdots + A_g x_g$.
Since we are mostly interested in the set $\cD_L$, there is no harm
in reducing to this case whenever $\cD_L\neq\emptyset$; see Exercise
\ref{ex:monicMe}.

  Let $p\in \RR^{\dd\times \dd}\ax$ be a given symmetric noncommutative
  $\dd\times\dd$-valued matrix polynomial. 
  Assuming that $p(0)\succ 0$, the
  positivity   set $\cD_{p}(n)$ of
  a noncommutative symmetric polynomial $p$ in dimension $n$
  is the component of $0$ of the set
\[
   \{ X \in(\SRnn)^g  \ \colon  \ p(X)\succ0 \}.
\]
  The
  \df{positivity
  set}, $\cD_p$, is the
  sequence of sets $(\cD_p(n) )_{n\in\NN}$.
   The noncommutative set $\cD_p$ is  called \df{convex} if, for each $n,$
  $\cD_p(n)$ is convex.

\begin{thm}[Helton-McCullough \cite{HMppt}]\label{thm:conv=lmi}
   Fix $p$ a $\dd \times \dd$  symmetric matrix
  of polynomials in  noncommuting variables.
 Assume
\begin{enumerate}[\rm (1)]
  \item $p(0)$ is positive definite;
  \item $\cD_p$ is bounded; and
  \item $\cD_p$ is convex.
 \end{enumerate}
Then there is a monic linear pencil $L$ such that
\[
\cD_L=\cD_p.
\]
\end{thm}

Here we shall confine ourselves to a few words about the techniques
involved in the proof, and refer the reader to \cite{HMppt}
for the full proof.   
Since we are dealing with matrix convex sets, it is not surprising
 that  the starting point for our analysis is the  matricial version
 of the Hahn-Banach Separation  theorem of
 Effros and Winkler \cite{EW97} which
 (itself a part of the theory of
 operator spaces and completely positive
 maps \cite{BL04,Pa02,Pi03})
 says that given a point $x$
 not inside a matrix convex set there is a (finite) linear matrix inequality which
 separates $x$ from the set.  For a general matrix convex
 set $\cC$, the conclusion is then that there is a collection,
 likely infinite, of LMIs which cut out $\cC$.

In the case $\cC$ is matrix convex and also semialgebraic,
   the challenge
   is to prove that there is actually
   a {\em finite} collection of LMIs which define $\cC$.
   The techniques used to meet this challenge
   have little relation to
   the methods of
   noncommutative
   calculus and positivity in the previous sections.
  Indeed a basic tool  (of independent
  interest) is  a degree bounded type
  of free Zariski closure of a single point $(X,v) \in \gtupn \times \RR^n$,
\[
   Z_d(X,v):= \bigcup_m \{ (Y, w) \in \gtupm\times \RR^m  :
    q(Y)w= 0 \text{ if } q(X)v = 0,  \; q \in \RR\ax_d \}.
\]
  Chief among a pleasant  list of natural properties is the fact 
  that there is an $(X,v)$ with $X\in\bd_p$ and  $p(X)v=0$ for which
  $Z_d(X,v)$ contains all pairs $(Y,w)$ such that
  $Y\in\bd_p$ and $p(Y)w=0$. 
  Combining this with the  Effros-Winkler Theorem
  and battling degeneracies is a bit tricky, but
  voila separation prevails in the end.
    See \cite{HMppt} for the details.

An unexpected  consequence of Theorem \ref{thm:conv=lmi} is that
projections of noncommutative semialgebraic sets may not be semialgebraic,
see Exercise \ref{ex:notproj}.
  For perspective, in the commutative case
  of a basic open semialgebraic subset $\cC$
  of $\RR^{g}$, 
  there is a stringent condition, called the
 ``\emph{line test}'' (see Chapter 6 for more details), 
 which, in addition to convexity,
  is  necessary for $\cC$ to be a spectrahedron.
  In two dimensions the line test is necessary and sufficient
  \cite{HV07}, a result used by Lewis-Parrilo-Ramana \cite{LPR05}
  to settle a 1958 conjecture of Peter Lax on hyperbolic polynomials.

 In summary, if a (commutative)
 bounded basic open semialgebraic convex set
 is a spectrahedron, then it must pass
 the highly restrictive line test; whereas
 a  nc basic open semialgebraic set is a spectrahedron
if and only if it is convex.

\subsection{Noncommutative Positivstellens\"atze under convexity assumptions}
\label{subsec:convPos}

 An algebraic certificate for positivity of a polynomial $p$
 on a semialgebraic set $S$ is a Positivstellensatz. The familiar
 fact that a polynomial $p$ in one-variable which is positive on
 $\RR$ is a sum of squares is an example.\index{Positivstellensatz}\index{sum of squares} 

 The theory of Positivstellens\"atze -  a pillar of the
 field of real algebraic geometry -  underlies
 the main approach currently used for global optimization
 of polynomials. 
 See \cite{Las} or Chapters 2 and 3 of Parrilo 
 for a beautiful treatment of this, and other,
 applications of
 commutative real algebraic geometry.\index{convex polynomial}
 Further, 
 because convexity of a polynomial
 $p$ on a set $S$ is equivalent to  positivity of the Hessian of
 $p$ on $S$, this theory also provides a link between convexity 
 and semialgebraic geometry.  Indeed, this link in the noncommutative
 setting ultimately lead to the conclusion the a matrix convex noncommutative polynomial has degree at most two, cf.~Section \ref{sec:eigsMM}.

In this section we give a result of opposite type.
We present a noncommutative Positivstellensatz for a polynomial
to be nonnegative on a convex semialgebraic nc set (i.e., on a spectrahedron).
Again, this result is 
cleaner and more
 rigid than the commutative counterparts (cf.~Theorem \ref{thm:ncsos}).

\begin{thm}[\cite{HKM++}]
Suppose $L$ is a monic linear pencil.\label{thm:lmipos}
Then a noncommutative polynomial $p$ is \emph{positive semidefinite}
on $\cD_L$ if and only if it
has a weighted sum of squares representation with optimal degree bounds.
Namely,
\beq\label{eq:wsos}
p
= s^\ss s   + \sum_j^{\rm finite} f_j^\ss L f_j,
\eeq
where
 $s, f_j$ are vectors of noncommutative polynomials of degree no greater than
 $\frac{\deg(p) }{2}$.
\end{thm}

The main ingredient of the proof is an analysis of
 rank preserving extensions of truncated
noncommutative
 Hankel matrices;
see \cite{HKM++} for details.
We point out that with $L=1$, Theorem \ref{thm:lmipos} recovers
Theorem \ref{thm:ncsos}.

Theorem \ref{thm:lmipos} contrasts sharply with the commutative setting,
where the degrees of $s, f_j$ are vastly greater than $\deg(p)$
 and assuming only $p$  nonnegative yields a clean
Positivstellensatz so seldom that the cases are noteworthy.

\subsection{Exercises}

\bexe\label{ex:monicMe}
Suppose $L$ is an affine linear pencil such that $0\in \cD_L(1)$.
Show that there is a monic linear pencil
$\check L$ with $\cD_L=\cD_{\check L}$.
\eexe
\index{linear pencil!monic}

\bexe\label{ex:notproj}
 Chapters 6 and 7
 discuss sets $D\subseteq \RR^g$ which have a semidefinite
 representation as a strict generalization 
  of a spectrahedron.   For instance, 
consider the TV screen $($cf.~Subsection {\rm\ref{subsec:ncsa}}$)$
\[
   {\rm ncTV}(1) =\{X\in\RR^2 \colon 1 - X_1^4 - X_2^4 > 0\}\subseteq \RR^2.
\]
Given $\alpha$ a positive real number,
  choose
  $\gamma^4= 1+2\alpha^2$ and
   let
\begin{equation}
  \label{eq:Lalpha0}
   L_0 = \begin{bmatrix} 1 & 0 &  y_1 \\
              0 & 1 & y_2 \\
               y_1 & y_2 &
                   1- 2\alpha (y_1+y_2) \end{bmatrix}
\end{equation}
  and
\begin{equation}
  \label{eq:Lalpha12}
   L_j =\begin{bmatrix} 1 & \gamma x_j
         \\ \gamma x_j & \alpha + y_j \end{bmatrix}, \quad j=1,2.
\end{equation}
   Note that the $L_j$ are not monic, but because $L_j(0)\succ 0,$
   they
   can be normalized to be monic without altering the solution
   sets of $L_j(X)\succ 0$, cf. Exercise \ref{ex:monicMe}.
Let $L=L_0\oplus L_1 \oplus L_2$.

 It is readily verified that ${\rm ncTV}(1)$ is the projection,
  onto the first two $($the $x)$  coordinates of the
 set $\cD_L(1)$; i.e., 
\[
{\rm ncTV}(1) =
\{ X \in \RR^2 \colon \exists Y\in\RR^2\  L(X,Y) \succ 0\}.
\]
\ben[\rm(1)]
\item
Show that ${\rm ncTV(1)}$ is not a spectrahedron.
$(${\em Hint}: How often is  $L_{\rm TV}(t X, tY)$ for $t \in \RR$  singular?$)$
\item
Show that ${\rm ncTV}$ is not the projection of the nc spectrahedron $\cD_L$.
\item
Show that ${\rm ncTV}$ is not the projection of \emph{any} nc spectrahedron.
\item
Is ${\rm ncTV}(2)$ a projection of a spectrahedron?
 $($Feel free to use the results about ncTV and LMI representable
 sets $($spectrahedra$)$, stated without proofs, from Subsection {\rm\ref{subsec:ncsa}}
 and Subsection {\rm\ref{subsec:lmi}.)}
\een
\eexe
\index{TV screen}
\index{spectrahedron}

\bexe
  If $q$ is a symmetric concave matrix-valued polynomial with $q(0)=I$,  then
  there exists a linear pencil $L$ and a matrix-valued linear polynomial
$\Lambda$ such that
  \[
   q = I-L -\Lambda^\ss \Lambda.
\]
\eexe
\index{polynomial!concave}

\bexe
  Consider the monic linear pencil
\[
  M(x) = \begin{bmatrix} 1 & x \\ x & 1\end{bmatrix}.
\]
\ben[\rm(1)]
\item
Determine $\cD_M$.
\item
Show that $1+x$ is positive semidefinite on $\cD_M$.
\item
Construct a representation for  $1+x$ 
 of the form \eqref{eq:wsos}.
\een
\eexe

\bexe
Consider the univariate affine linear pencil
\[
L(x)= \begin{bmatrix} 1 & x \\ x & 0\end{bmatrix}.
\]
\ben[\rm(1)]
\item
Determine $\cD_L$.
\item
Show that $x$ is positive semidefinite on $\cD_L$.
\item
Does $x$ 
admit 
a representation of the form \eqref{eq:wsos}?
\een
\eexe

\bexe
Let $L$ be an affine linear pencil. Prove that: 
\ben[\rm (1)]
\item
$\cD_L$ is bounded if and only if  $\cD_L(1)$ is bounded;
\item
$\cD_L=\emptyset$ if and only if $\cD_L(1)=\emptyset$.
\een
\eexe

\bexe
Let $L=I+A_1x_1+\cdots+A_gx_g$ be a monic linear pencil and
assume that $\cD_L(1)$ is bounded. 
Show that $I,A_1,\ldots,A_g$ are linearly independent.
\eexe

\bexe
Let
$$
\De(x_1,x_2)=
I+ 
 \begin{bmatrix}  0 & 1 & 0 \\
                     1 & 0 & 0\\
                     0 & 0 & 0
                      \end{bmatrix} x_1 +
\begin{bmatrix}  0 & 0 & 1 \\
                     0 & 0 & 0\\
                     1 & 0 & 0
                      \end{bmatrix} x_2 =
\begin{bmatrix}  1 & x_1 & x_2 \\
                     x_1 & 1 & 0\\
                     x_2 & 0 & 1
                      \end{bmatrix}$$
and
$$
\Ga(x_1,x_2)=I+ \begin{bmatrix}1  & 0 \\ 0 & -1
\end{bmatrix} x_1 +  \begin{bmatrix} 0  & 1 \\ 1 & 0
\end{bmatrix} x_2 =  \begin{bmatrix} 1+x_1  & x_2 \\ x_2 & 1-x_1
\end{bmatrix}
$$
be affine linear pencils. Show:
\ben[\rm (1)]
\item
$\cD_{\De}(1)= \cD_{\Ga}(1)$.
\item
$\cD_{\Ga}(2)\subsetneq \cD_{\De}(2)$.
\item 
Is $\cD_{\De}\subseteq\cD_{\Ga}$?
What about
$\cD_{\Ga}\subseteq\cD_{\De}$?
\een
\eexe

\bexe
Let $L=A_1x_1+\cdots + A_gx_g\in \SS^{d\times d}\ax$ be a (homogeneous) linear
pencil. Then the following are equivalent:
\ben[\rm(i)]
\item
$\cD_L(1)\neq\emptyset$;
\item If $u_1,\dots,u_m\in\RR^d$ with $\sum_{i=1}^mu_i^\ss L(x)u_i=0$,
then $u_1=\dots=u_m=0$.
\een
\eexe

\section{From free real algebraic geometry to the real world}

Now that you have gone through the mathematics
we return to its implications.
In the linear systems engineering problems you have seen
both in Subsection \ref{subsec:Motivation} and in Chapter 2.2.1,
the conclusion was that the problem was equivalent to
solving an LMI. Indeed this is what one sees throughout the
literature. Thousands of engineering papers have a dimension
free problem and it converts (often by serious
cleverness) to an LMI in the best of cases,
or more likely there is some approximate solution which is
an LMI.

  While
engineers would be satisfied with convexity,  what they actually
do get is an LMI.
One would hope that there is a rich world of convex situations
not equivalent to an LMI. Then there would be  a variety of methods
waiting to be discovered for dealing with them.
Alas what we have shown here is compelling evidence
that any convex dimension free problem is equivalent to an LMI.
Thus there is no rich world  of convexity beyond what is already known
and no armada of techniques beyond those for producing LMIs which we
already see all around us.

\printindex
\end{document}